\documentclass[a4paper,11pt,twoside,onecolumn]{article}

\usepackage[mac]{inputenc}
 \usepackage[T1]{fontenc}
 \usepackage[normalem]{ulem}
 \usepackage[english]{babel}
\usepackage{amsmath}
  \usepackage{amsthm}
 \usepackage{bbm}
 \usepackage{amssymb}
 \usepackage{verbatim}
 \usepackage{vmargin}
 \usepackage{graphicx}
 \usepackage{color}
 \usepackage{epsfig}
 \usepackage[only,llbracket,rrbracket]{stmaryrd}
 \usepackage{enumerate}

\newtheorem{thm}{Theorem}[section]
\newtheorem*{deft}{Definition}
\newtheorem*{notat}{Notations}
\newtheorem{prop}{Proposition}[section]
\newtheorem{lem}{Lemma}[section]

\numberwithin{equation}{section}

\theoremstyle{theorem}
\newtheorem{rques}{\textbf{Remarks}}[section]{\vskip 0.5cm}

\newtheorem{rque}{\textbf{Remark}}[section]{\vskip 0.5cm} 
\newtheorem*{ack}{\textbf{Acknowledgments}}{\vskip 0.5cm} 

\title{On the three-dimensional finite Larmor radius approximation:  the case of electrons in a fixed background of ions} 
\author{Daniel Han-Kwan\footnote{\'Ecole Normale Sup\'erieure, D\'epartement de Math\'ematiques et Applications,  45 rue d'Ulm 75230 Paris Cedex 05 France, email : {hankwan@dma.ens.fr}}}
\date{}
\begin{document}
 \maketitle

\begin{abstract}
This paper is concerned with the analysis of a mathematical model arising in plasma physics, more specifically in fusion research. It directly follows \cite{DHK1}, where the tri-dimensional analysis of a Vlasov-Poisson equation with finite Larmor radius scaling was led, corresponding to the case of ions with massless electrons whose density follows a linearized Maxwell-Boltzmann law. We now consider the case of electrons in a background of fixed ions, which was only sketched in \cite{DHK1}. Unfortunately, there is evidence that the formal limit is false in general. Nevertheless, we formally derive a fluid system for particular monokinetic data. We prove the local in time existence of analytic solutions and rigorously study the limit (when the Debye length vanishes) to a new anisotropic fluid system. This is achieved thanks to Cauchy-Kovalevskaya type techniques, as introduced by Caflisch \cite{Caf} and Grenier \cite{Gre1}. We finally show that this approach fails in Sobolev regularity, due to multi-fluid instabilities.
\end{abstract}

\textbf{Keywords}: Finite Larmor Radius Approximation - Anisotropic quasineutral limit - Anisotropic hydrodynamic systems - Cauchy-Kovalevskaya theorem - Ill-posedness in Sobolev spaces.

\section{Introduction}
\subsection{Presentation of the problem}

The main goal of this paper is to derive some fluid model in order to understand the behaviour of a quasineutral gas of electrons in a neutralizing background of fixed ions and submitted to a strong magnetic field. For simplicity, we consider that the magnetic field has fixed direction and intensity. The density of the electrons is governed by the classical Vlasov-Poisson equation.
We first introduce some notations:

\begin{notat}

\begin{itemize}
Let $(e_1,e_2,e_\parallel)$ be a fixed orthonormal basis of $\mathbb{R}^3$.
\item The subscript $\perp$ stands for the orthogonal projection on the plane $(e_1,e_2)$, while the subscript $\parallel$ stands for the projection on $e_\parallel$ .
\item For any vector $X=(X_1,X_2,X_\parallel)$, we define $X^\perp$ as the vector $(X_y,-X_x,0)=X\wedge e_\parallel$.
\item We define the differential operators $\Delta_{x_\parallel}= \partial_{x_\parallel}^2$ and $\Delta_{x_\perp}=\partial^2_{x_1} + \partial^2_{x_2}$.
\end{itemize}
\end{notat}

The scaling we consider (we refer to the appendix for physical explanations) leads to the study of the scaled Vlasov-Poisson system (for $t>0, x \in \mathbb{T}^3:= \mathbb{R}^3/ \mathbb{Z}^3, v \in \mathbb{R}^3$):

  \begin{equation}
  \label{kinbegin}
\left\{
    \begin{array}{ll}
    \partial_{t} f_\epsilon + \frac{v_\perp}{\epsilon}.\nabla_{x} f_\epsilon + v_\parallel.\nabla_{x} f_\epsilon + (E_\epsilon+ \frac{v\wedge e_\parallel}{\epsilon}).\nabla_{v} f_\epsilon = 0 \\
  E_\epsilon= (-\nabla_{x_\perp} {V}_\epsilon, -\epsilon\nabla_{x_\parallel} {V}_\epsilon)  \\
  -\epsilon^2\Delta_{x_\parallel} {V}_\epsilon -\Delta_{x_\perp} {V}_\epsilon = \int f_\epsilon dv - \int f_\epsilon dvdx\\
     f_{\epsilon, t=0}=f_{\epsilon,0}\geq 0, \quad \int f_{\epsilon,0}dvdx=1.
\end{array}
  \right.
\end{equation}
The quantity $f_\epsilon(t,x,v)$ is interpreted as the distribution function of the electrons: this means that $f_\epsilon(t,x,v) dx dv$ is the probability of finding particles at time $t$ with position $x$ and velocity $v$; $V_\epsilon(t,x)$ and $E_\epsilon(t,x)$ are respectively the electric potential and force.

This corresponds to the so-called finite Larmor radius scaling for the Vlasov-Poisson equation, which was introduced by Frénod and Sonnendrücker in the mathematical literature \cite{FS2}. The $2D$ version of the system (obtained when one restricts to the perpendicular dynamics) was studied in \cite{FS2} and more recently in \cite{Bos} and \cite{GHN}. A version of the full $3D$ system describing ions with massless electrons was studied by the author in \cite{DHK1}. In this work, we considered that the density of electrons  follows a linearized Maxwell-Boltzmann law. This means that we studied the following Poisson equation for the electric potential:
\begin{equation}
V_\epsilon  -\epsilon^2\Delta_{x_\parallel} {V}_\epsilon -\Delta_{x_\perp} {V}_\epsilon = \int f_\epsilon dv - \int f_\epsilon dvdx.
\end{equation}
In this case it was shown after some filtering that the number density $f_\epsilon$ weakly-* converges to some solution $f$ to another kinetic system exhibiting the so-called $E\times B$ drift in the orthogonal plane, but with trivial dynamics in the parallel direction. This last feature seems somehow disappointing.

We observed in \cite{DHK1} that in the case where the Poisson equation reads:
\begin{equation}
  -\epsilon^2\Delta_{x_\parallel} {V}_\epsilon -\Delta_{x_\perp} {V}_\epsilon = \int f_\epsilon dv - \int f_\epsilon dvdx,
  \end{equation}
 we could expect to make a pressure appear in the limit process $\epsilon \rightarrow 0$, due to the incompressibility constraint:
  \[
  \int f dv dx_\perp =\int f dvdx.
\]
  
Unfortunately, we were not able to rigorously derive a kinetic limit or even a fluid limit  from (\ref{kinbegin}). This is not only due to technical mathematical difficulties. This is related to the existence of instabilities for the Vlasov-Poisson equation, such as the double-humped instabilities (see Guo and Strauss \cite{GS}) and their counterpart in the multi-fluid Euler equations, such as the two-stream instabilities (see Cordier, Grenier and Guo \cite{CGG}). Such instabilities actually take over in the limit $\epsilon \rightarrow 0$ and the formal limit is false in general, unless $f_{\epsilon,0}$ does not depend on parallel variables, which corresponds to the $2D$ problem studied by Frénod and Sonnendrücker \cite{FS2}. 

Actually, we can observe that if on the contrary the initial data $f_{\epsilon,0}$ depends only on parallel variables, we obtain the one-dimensional quasineutral system:
  \begin{equation}
\left\{
    \begin{array}{ll}
    \partial_{t} f_\epsilon + v_\parallel \partial_{x_\parallel} f_\epsilon -\partial_{x_\parallel} V_\epsilon \partial_{v_\parallel} f_\epsilon = 0 \\
  -\epsilon^2\partial^2_{x_\parallel} {V}_\epsilon = \int f_\epsilon dv - \int f_\epsilon dvdx_\parallel\\
     f_{\epsilon, t=0}=f_{\epsilon,0}\geq 0, \quad \int f_{\epsilon,0}dvdx_\parallel=1.
\end{array}
  \right.
\end{equation}

The formal limit is easily obtained, by taking $\epsilon=0$:
  \begin{equation}
\left\{
    \begin{array}{ll}
    \partial_{t} f + v_\parallel \partial_{x_\parallel} f -\partial_{x_\parallel} V \partial_{v_\parallel} f = 0 \\
  -\epsilon^2\partial^2_{x_\parallel} {V} = \int f dv - \int f dvdx_\parallel\\
     f_{ t=0}=f_{0}\geq 0, \quad \int f_{0}dvdx_\parallel=1.
\end{array}
  \right.
\end{equation}

In \cite{Gr99}, an explicit example of Grenier shows that the formal limit is false in general, because of the double-humped instability:

\begin{thm}[\cite{Gr99}]For any $N$ and $s$ in $\mathbb{N}$, and for any $\epsilon<1$, there exist for $i=1,2,3,4$, $v_i^\epsilon(x)\in H^s(\mathbb{T})$ with $\Vert v_1^\epsilon(x)+1 \Vert_{H^s} \leq \epsilon^N$, $\Vert v_2^\epsilon(x)+1/2 \Vert_{H^s} \leq \epsilon^N$, $\Vert v_3^\epsilon(x)-1/2 \Vert_{H^s} \leq \epsilon^N$, $\Vert v_4^\epsilon(x)-1 \Vert_{H^s} \leq \epsilon^N$, such that the solution $f_\epsilon(t,x,v)$ associated to the initial data defined by:
\begin{eqnarray*}
f_{\epsilon,0}(x,v)&=&1 \quad \text{for} \quad v_1^\epsilon(x) \leq v \leq v_2^\epsilon(x) \text{      and      } v_3^\epsilon(x) \leq v \leq v_4^\epsilon(x) \\
&=& 0 \quad \text{elsewhere}.
\end{eqnarray*}
We also define $f_0$ by:
\begin{eqnarray*}
f_{0}(x,v)&=&1 \quad \text{for} \quad -1 \leq v \leq -1/2 \text{      and      }1/2  \leq v \leq 1 \\
&=& 0 \quad \text{elsewhere}.
\end{eqnarray*}

Then $f_\epsilon$ does not converge to $f_0$ in the following sense:
\begin{equation}
\liminf_{\epsilon\rightarrow 0} \sup_{t\leq T} \int \vert f_\epsilon(t,x,v) - f_0(v) \vert v^2 dv dx > 0
\end{equation}
for any $T>0$ and also for $T= \epsilon^\alpha$, with $\alpha<1/2$.
\end{thm}

In order to overcome the effects of these instabilities for the usual quasineutral limit, there are two possibilities:
\begin{itemize}
\item One consists in restricting to particular initial profiles chosen in order to be stable (this would imply in particular some monotony conditions on the data, such as the Penrose condition \cite{Pen}).
\item The other one consists in considering data with analytic regularity, in which case the instabilities ( which are essentially of ``Sobolev'' nature) do not have any effect.
\end{itemize}

Here the situation is worst: by opposition to the usual quasineutral limit (see \cite{Br}, \cite{Gr99}), restricting to stable profiles is not sufficient. This is due to the anisotropy of the problem and the dynamics in the perpendicular variables.

In this paper, we illustrate this phenomenon by  formally deriving the following fluid system, obtained from the kinetic system (\ref{kinbegin}) by considering some physically relevant monokinetic data  (we refer to the appendix for the detailed formal derivation).

\begin{equation}
\label{sys}
 \left\{
    \begin{array}{ll}
  \partial_t \rho_\epsilon + \nabla_\perp (E^\perp_\epsilon \rho_\epsilon) + \partial_\parallel(v_{\parallel,\epsilon} \rho_\epsilon)= 0 \\
  \partial_t v_{\parallel,\epsilon} + \nabla_\perp (E^\perp_\epsilon v_{\parallel,\epsilon}) + v_{\parallel,\epsilon} \partial_\parallel(v_{\parallel,\epsilon}) = -\epsilon\partial_\parallel \phi_\epsilon(t,x) -\partial_\parallel V_\epsilon(t,x_\parallel) \\
  E^\perp_\epsilon= -\nabla^\perp \phi_\epsilon \\
-\epsilon^2 \partial^2_{\parallel} \phi_\epsilon - \Delta_{\perp} \phi_\epsilon = \rho_\epsilon - \int \rho_\epsilon dx_\perp\\
  -\epsilon \partial_{\parallel}^2 V_\epsilon =\int \rho_\epsilon dx_\perp -1,\\
\end{array}
  \right.
\end{equation}
where:

\begin{itemize}\item$\rho_\epsilon(t,x_\perp,x_\parallel) : \mathbb{R}^+ \times \mathbb{T}^3 \rightarrow \mathbb{R}^+_*$ can be interpreted as a charge density,

\item $v_{\parallel,\epsilon}(t,x_\perp,x_\parallel)  : \mathbb{R}^+ \times \mathbb{T}^3 \rightarrow \mathbb{R}$ can be interpreted as a ``parallel'' current density. 
 
\item $\phi_\epsilon(t,x_\parallel)$ and $V_{\epsilon}(t,x)$ are electric potentials.

\end{itemize}

Although we have considerered monokinetic data, (\ref{sys}) is intrinsically a ``multi-fluid'' system, because of the dependence on $x_\perp$. Hence, we still have to face the two-stream instabilities (\cite{CGG}): because of these, the limit is false in Sobolev regularity and we thus decide to study the associated Cauchy problem for analytic data. 

We then prove the limit to a new fluid system which is strictly speaking compressible but also somehow ``incompressible in average''. This rather unusual feature is due to the anisotropy of the model. The fluid system is the following (obtained formally by taking $\epsilon=0$):
\begin{equation}
\label{simple}
\left\{
    \begin{array}{ll}
  \partial_t \rho + \nabla_\perp (E^\perp \rho) + \partial_\parallel(v_\parallel \rho)= 0 \\
  \partial_t v_\parallel + \nabla_\perp (E^\perp v_\parallel) + v_\parallel \partial_\parallel(v_\parallel) = -\partial_\parallel p(t,x_\parallel) \\
  E^\perp= \nabla^\perp \Delta_\perp^{-1} \left(\rho - \int \rho dx_\perp\right)\\
  \int \rho dx_\perp = 1.\\
\end{array}
  \right.
\end{equation}

We observe that this system can be interpreted as an infinite system of Euler-type equations, coupled together through the ``parameter'' $x_\perp$. It has some interesting features:

\begin{itemize}
 \item This system is highly anisotropic in $x_\perp$ and $x_\parallel$. The $2D$ part of the dynamics of the equation for $\rho$ is nothing but the vorticity formulation of $2D$ incompressible Euler. Physically speaking, $\rho$ should be interpreted here as a density rather than a vorticity. The dynamics in the parallel direction is similar to the dynamics of incompressible Euler written in velocity. We finally observe that the pressure $p$ only depends on the parallel variable $x_\parallel$ and not on $x_\perp$.
 \item This does not strictly speaking describe an incompressible fluid, since $(E^\perp, v_\parallel)$ is not divergence free. Somehow, the fluid is hence compressible. But the constraint $ \int \rho dx_\perp = 1$ can be interpreted as a constraint of ``incompressibility in average'' which allows one to recover the pressure law from the other unknowns. Indeed, we easily get, thanks to the equation on $\rho$: 
\begin{equation}
 \partial_{x_\parallel} \int \rho v_\parallel dx_\perp =0.
\end{equation}
So by plugging this constraint in the equation on $\rho v_\parallel$:
\begin{equation*}
 \partial_t (\rho v_\parallel) + \nabla_\perp (E^\perp \rho_\parallel v_\parallel) +  \partial_\parallel(\rho v_\parallel^2) = -\partial_\parallel p(t,x_\parallel)\rho,
\end{equation*}
we get the (one-dimensional !) elliptic equation allowing to recover $-\partial_{x_\parallel} p$:
\begin{equation*}
 - \partial^2_\parallel p(t,x_\parallel) = \partial^2_\parallel \int \rho v_\parallel^2 dx_\perp,
\end{equation*}
from which we get:
\begin{equation}
 - \partial_\parallel p(t,x_\parallel) = \partial_\parallel \int \rho v_\parallel^2 dx_\perp.
\end{equation}
\item From the point of view of plasma physics, $E^\perp.\nabla_\perp$ is the so-called electric drift. By analogy with the so-called drift-kinetic equations \cite{Wes}, we can call this system a drift-fluid equation. To the best of our knowledge, this is the very first time such a model is exhibited in the literature.

\end{itemize}

From now on, when there is no risk of confusion, we will sometimes write $v$ and $v_\epsilon$ instead of $v_\parallel$ and $v_{\parallel,\epsilon}$.

\subsection{Organization of the paper}

The outline of this paper is as follows. In Section \ref{sec-results}, we will state the main results of this paper that are: the existence of analytic solutions to (\ref{sys}) locally in time but uniformly in $\epsilon$ (Theorem \ref{exi}), the strong convergence to (\ref{simple}) with a complete description of the plasma oscillations (Theorem \ref{con}) and finally the existence and uniqueness of local analytic solutions to (\ref{simple}), in Proposition \ref{exi2}.

Section \ref{sec-proof1} is devoted to the proof of Theorem \ref{exi}. First we recall some elementary features of the analytic spaces we consider (section \ref{sec-func}), then we implement an approximation scheme for our Cauchy-Kovalesvkaya type existence theorem. The results are based on a decomposition of the electric field allowing for a good understanding of the so-called plasma waves (section \ref{sec-waves}).

In section \ref{sec-proof2}, we prove Theorem \ref{con}, by using the uniform in $\epsilon$ estimates we have obtained in the previous theorem. The proof relies on another decomposition of the electric field, in order to exhibit the effects of the plasma waves as $\epsilon$ goes to $0$.

Then, in section \ref{sec-sharp}, we discuss the sharpness of our results:
\begin{itemize}
\item In sections \ref{sec-analytic1} and \ref{sec-analytic2}, we discuss the analyticity assumption and explain why we can not lower down the regularity to Sobolev. In section \ref{sec-local}, we explain why it is not possible to obtain global in time results. We obtain these results  by considering some well-chosen initial data and using results of Brenier on multi-fluid Euler systems \cite{Br3}.

\item Because of the multi-stream instabilities, studying the limit with the relative entropy method is bound to fail. Nevertheless we found it interesting to try to apply the method and see at which point things get nasty: this is the object of section \ref{sec-rela}, where we study a kinetic toy model which retains the main unstable feature of system (\ref{sys}). 

\end{itemize}

The two last sections are respectively a short conclusion and an appendix where we explain the scaling and the formal derivation of system (\ref{sys}).

\section{Statement of the results}
\label{sec-results}
In order to prove both the existence of strong solutions to systems (\ref{sys}) and (\ref{simple}) and also prove the results of convergence, we follow the construction of Grenier \cite{Gre1}, with some modifications adapted to our problem. 

In \cite{Gre1}, Grenier studies the quasineutral limit of the family of coupled Euler-Poisson systems:

\begin{equation}
\label{coupledeuler}
\left\{
    \begin{array}{ll}
  \partial_t \rho_\Theta^\epsilon  + \operatorname{div}( \rho_\Theta^\epsilon v_\Theta^\epsilon)= 0 \\
  \partial_t v_\Theta^\epsilon +  v_\Theta^\epsilon.\nabla(v_\Theta^\epsilon) = E^\epsilon \\
  \operatorname{rot} E^\epsilon = 0 \\
  \epsilon \operatorname{div}E^\epsilon = \int_M \rho_\Theta^\epsilon \mu(d\Theta)  - 1,\\
\end{array}
  \right.
\end{equation}
with $(M, \Theta,  \mu)$ a probability space.

Following the proof of the Cauchy-Kovalevskaya theorem given by Caflisch \cite{Caf}, Grenier proved the local existence of analytic functions (with respect to $x$) uniformly with respect to $\epsilon$ and then, after filtering the fast oscillations due to the force field, showed the strong convergence  to the system:
\begin{equation}
\label{grenierS}
\left\{
    \begin{array}{ll}
  \partial_t \rho_\Theta  + \operatorname{div}( \rho_\Theta v_\Theta)= 0 \\
  \partial_t v_\Theta +  v_\Theta^\epsilon.\nabla(v_\Theta) = E \\
  \operatorname{rot} E = 0 \\
  \int \rho_\Theta \mu(d\Theta)  = 1.\\
\end{array}
  \right.
\end{equation}

We notice that the class of systems studied by Grenier is close to system (\ref{sys}), if we take  $x=x_\parallel$, $\Theta=x_\perp$ and $(M, \mu)=(\mathbb{T}^2, dx_\perp)$, the main difference being that we have to deal with a dynamics in $\Theta=x_\perp$.

Hence, we introduce the same spaces of analytic functions as in \cite{Gre1}, but this time depending also on $\Theta=x_\perp$.

\begin{deft}
 Let $\delta>1$. We define $B_\delta$ the space of real functions $\phi$ on $\mathbb{T}^3$ such that
\begin{equation}
\label{norm}
 \vert \phi \vert_\delta= \sum_{k \in \mathbb{Z}^3} \vert \mathcal{F} \phi (k) \vert \delta^{\vert k \vert} < +\infty,
\end{equation}
where $\mathcal{F}\phi(k)$ is the k-th Fourier coefficient of $\phi$ defined by:
$$\mathcal{F}\phi(k)=  \int_{\mathbb{T}^3} \phi(x) e^{-i2\pi k.x} dx.$$
\end{deft}

The first theorem proves the existence of local analytic solutions of (\ref{sys}) with a life span uniform in $\epsilon$. 

\begin{thm}
\label{exi}
 Let $\delta_0>1$. Let $\rho_\epsilon (0)$ and $v_\epsilon (0)$ be two bounded families of $B_{\delta_0}$ such that $\int \rho_\epsilon(0) dx=1$ and:
\begin{equation}
 \left\Vert \int \rho_\epsilon(0) dx_\perp - 1\right\Vert_{B_{\delta_0}} \leq C \sqrt{\epsilon},
\end{equation}
then there exists $\eta>0$ such that for every $\delta_1 \in ]1, \delta_0[$, for any $\epsilon>0$, there exists a unique strong solution $(\rho_\epsilon, v_\epsilon)$ to (\ref{sys})  bounded uniformly in $\mathcal{C}([0,\eta(\delta_0-\delta_1)[,B_{\delta_1})$  with initial conditions $(\rho_\epsilon (0), v_\epsilon (0))$.
Moreover, $\sqrt{\epsilon}\partial_\parallel  V_\epsilon$ is uniformly bounded in $\mathcal{C}([0,\eta(\delta_0-\delta_1)[,B_{\delta_1})$.
\end{thm}

\begin{rque}
\begin{itemize}

  \item The condition $\left\Vert \int \rho_\epsilon(0) dx_\perp - 1\right\Vert_{B_{\delta_0}} \leq C \sqrt{\epsilon}$ implies that $\sqrt{\epsilon}\partial_\parallel V_\epsilon(0)$ is bounded uniformly in $B_{\delta_0}$ (this is the correct scale in view of the energy conservation).

 \item Note that for all $t\geq0, \int \rho_\epsilon dx=1$. Hence the Poisson equation $ -\epsilon \partial_{\parallel}^2 V_\epsilon =\int \rho_\epsilon dx_\perp -1$ can always be solved.
\end{itemize}
\end{rque}

We can then prove the convergence result:

\begin{thm}
 \label{con}
Let $(\rho_\epsilon, v_\epsilon)$ be solutions to the system (\ref{sys}) for $0\leq t \leq T$ satisfying for some $s>7/2$ the following estimate:
\begin{equation}
 (H): \sup_{t\leq T, \epsilon} \left( \Vert \rho_\epsilon \Vert_{H^s_{x_\perp,x_\parallel}}+\Vert v_\epsilon \Vert_{H^s_{x_\perp,x_\parallel}}+ \Vert \sqrt{\epsilon}\partial_{x_\parallel} V_\epsilon \Vert_{H^s_{x_\parallel}} \right) < +\infty .
\end{equation}
Then we get the following convergences
$$\rho_\epsilon \rightarrow \rho, $$
$$ v_\epsilon - \frac{1}{i}(  E_+ e^{it/\sqrt{\epsilon}} - E_- e^{-it/\sqrt{\epsilon}}) \rightarrow v, $$
strongly respectively in $\mathcal{C}([0,T], H^{s'}_{x_\perp, x_\parallel})$ and $\mathcal{C}([0,T], H^{s'-1}_{x_\perp, x_\parallel})$ for all $s'<s$,
and
$$ \sqrt{\epsilon}\left(-\partial_{x_\parallel} V_\epsilon - ( E_+ e^{it/\sqrt{\epsilon}} + E_- e^{-it/\sqrt{\epsilon}})\right)\rightarrow 0, $$
strongly in $\mathcal{C}([0,T], H^{s'}_{x_\parallel})$ for all $s'<s-1$, and where $(\rho, v)$ is solution to the asymptotic system (\ref{simple}) on $[0,T]$ with initial conditions:
$$\rho(0)=\lim_{\epsilon\rightarrow 0} \rho_\epsilon(0),$$
$$v(0) =\lim_{\epsilon\rightarrow 0}\left( v_\epsilon(0) - \int \rho_\epsilon  v_\epsilon dx_\perp (0)\right)  $$
and $E_+(t,x_\parallel), E_-(t,x_\parallel)$ are gradient correctors which satisfy the transport equations:
$$\partial_t  E_\pm + \left(\int \rho v dx_\perp\right) \partial_{x_\parallel}  E_\pm =0,$$
with initial data:
$$\partial_{x_\parallel} E_ +(0) = \lim_{\epsilon\rightarrow 0} \frac{1}{2} \partial_{x_\parallel} \left(-\sqrt{\epsilon}\partial_{x_\parallel}V_\epsilon(0) + i\int \rho_\epsilon v_\epsilon dx_\perp(0) \right), $$
$$\partial_{x_\parallel} E_ -(0) = \lim_{\epsilon\rightarrow 0}\frac{1}{2} \partial_{x_\parallel}\left(-\sqrt{\epsilon}\partial_{x_\parallel}V_\epsilon(0) - i\int \rho_\epsilon v_\epsilon dx_\perp(0) \right). $$
\end{thm}

As explained in the introduction, due to the two-streams instabilities, we have to restrict to data with analytic regularity: the Sobolev version of these results is false in general (see \cite{CGG} and the discussion of Section \ref{sec-sharp}). 

\begin{rque}
 \begin{itemize}
  \item It is clear that solutions built in Theorem \ref{exi} satisfy $(H)$.
  \item If instead of $(H)$ we make the stronger assumption, for $\delta>1$
\begin{equation}
 (H'): \sup_{t\leq T, \epsilon} \left( \Vert \rho_\epsilon \Vert_{B_\delta}+\Vert v_\epsilon \Vert_{B_\delta}+ \Vert \sqrt{\epsilon}\partial_{x_\parallel} V_\epsilon \Vert_{B_\delta} \right) < +\infty,
\end{equation}
then we get the same strong convergences in $\mathcal{C}([0,T], B_{\delta'})$ for all $\delta'<\delta$.

Using Lemma \ref{lemma} $(ii), (iv)$, the proof under assumption $(H')$ is the same as under assumption $(H)$. 

\item The ``well-prepared'' case corresponds to the case when:
$$\lim_{\epsilon\rightarrow 0} -\sqrt{\epsilon}\partial^2_{x_\parallel}V_\epsilon(0) =0,$$
$$\lim_{\epsilon\rightarrow 0} \partial_{x_\parallel}\int \rho_\epsilon v_\epsilon dx_\perp(0)=0. $$
Then there is no corrector.
\end{itemize}

\end{rque}

With the same method used for Theorem \ref{exi}, we can also prove a theorem of existence and uniqueness of analytic solutions to system (\ref{simple}). 
\begin{prop}
\label{exi2}
 Let $\delta_0>\delta_1>1$. For initial data $\rho(0), v(0) \in B_{\delta_0}$ satisfying
 \begin{equation}
\rho(0) \geq 0,
\end{equation}
\begin{equation}
\int \rho(0) dx_\perp=1
\end{equation}
and 
\begin{equation}
\partial_\parallel \int  \rho(0)v(0) dx_\perp=0,
\end{equation}
there exists $\eta>0$ depending on $\delta_0$ and on the initial conditions only such that there is a  unique strong solution $(\rho, v_\parallel, p)$ to the system (\ref{simple}) with  $\rho,v \in \mathcal{C}([0,\eta(\delta_0-\delta_1)[,B_{\delta})$ for all $\delta<\delta_1$.
\end{prop}

\section{Proof of Theorem \ref{exi}}
\label{sec-proof1}
\subsection{Functional analysis on $B_\delta$ spaces}
\label{sec-func}
First we define the time dependent analytic spaces we will work with.

Let $\beta$ be an arbitrary constant in $]0,1[$ (take for instance $\beta=1/2$ to fix ideas) and $\eta>0$ a parameter to be chosen later.

\begin{deft}Let $\delta_0>1$. We define the space
$B_{\delta_0}^\eta=\{u \in \mathcal{C}^0 ([0, \eta(\delta_0-1)], B_{\delta_0-t/\eta}) \}$,
endowed with the norm
 $$\Vert u \Vert_{\delta_0} = \sup_{\left\{
    \begin{array}{ll}
1<\delta\leq \delta_0 \\
0\leq t\leq \eta(\delta_0-\delta)
\end{array}\right.} \left( \vert u(t)\vert_\delta + (\delta_0-\delta- \frac{t}{\eta})^\beta\vert \nabla u(t)\vert_\delta)\right),  $$
where the norm $\vert u\vert_\delta$ was defined in (\ref{norm}):
\[
 \vert u \vert_\delta= \sum_{k \in \mathbb{Z}^3} \vert \mathcal{F} u (k) \vert \delta^{\vert k \vert} < +\infty,
\]
\end{deft}

We now gather from \cite{Gre1} a few elementary properties of these spaces, that we recall for the reader's convenience.
\begin{lem}
\label{lemma}
 \begin{enumerate}[(i)]For all $\delta>1$
  \item The spaces $B_\delta$ and $B_\delta^\eta$ are Banach algebra.
\item If $\delta'<\delta$ then $B_{\delta} \subset B_{\delta'}$, the embedding being continuous and compact.
\item For all $s \in \mathbb{R}$, $B_\delta \subset H^s$, the embedding being continuous and compact.
\item  For all $1<\delta'<\delta$, if $\phi \in B_\delta$,
$$\vert\nabla \phi\vert_{\delta'} \leq \frac{\delta}{\delta-\delta'} \vert\phi\vert_\delta.$$
\item If $u$ is in $B_{\delta_0} ^\eta$ and if $\delta+t/\eta<\delta_0$ then
$$\vert \partial^2_{x_i, x_j} u(t)\vert_\delta  \leq 2^{1+\beta}\Vert u\Vert_{\delta_0} \delta_0 (\delta_0-\delta- \frac{t}{\eta})^{-\beta-1}.$$
 \end{enumerate}
\end{lem}

For further properties of these spaces we refer to the recent work of Mouhot and Villani \cite{MV}, in which similar analytic spaces (and more sophisticated versions) are considered.

\begin{proof}
 We give an elementary proof for $(ii)$ which is not given in \cite{Gre1} .
The embedding is obvious. We consider for $N\in \mathbb{N}$ the map $i_N$ defined by:
$$i_N(\phi)= \sum_{\vert k \vert \leq N}\mathcal{F}\phi(k)e^{i xk}. $$
We then compute:
$$ \vert (Id - i_N) \phi \vert_{\delta'} = \sum_{\vert k \vert > N}\vert \mathcal{F}\phi(k)\vert \delta'^{\vert k \vert} 
\leq  \left(\frac{\delta'}{\delta}\right)^N \sum_{\vert k \vert > N}\vert \mathcal{F}\phi(k)\vert \delta^{\vert k \vert} 
\leq  \left(\frac{\delta'}{\delta}\right)^N \vert \phi \vert_\delta.$$

So the embedding $B_{\delta} \subset B_{\delta'}$ is compact as the limit of finite rank operators.

For $(v)$, take $\delta'=\delta+\frac{\delta_0-\delta-t/\eta}{2}$ and apply $(iv)$. We refer to  \cite{Gre1} for the other proofs.
\end{proof}

We will also need the following elementary observation:

\begin{rque}
 \label{rque}
Let $\phi \in B_\delta$. Then:
$$\left\vert \int \phi dx_\perp \right\vert_\delta \leq \vert \phi \vert_\delta.$$
\end{rque}

\begin{proof}We simply compute:
 $$\left\vert \int \phi dx_\perp \right\vert_\delta= \sum_{k_\perp=0, k_\parallel \in \mathbb{N}} \vert \mathcal{F}(\phi)(k_\perp,k_\parallel) \vert \delta^{\vert k\vert} \leq \sum_ {k \in \mathbb{N}^3} \vert \mathcal{F}(\phi)\vert \delta^{\vert k\vert}=\vert \phi \vert_\delta.$$
\end{proof}

\subsection{Description of plasma oscillations}
\label{sec-waves}

To simplify notations, we set $E_{\epsilon,\parallel} = -\partial_{x_\parallel} V_\epsilon(t,x_\parallel)$ (which has nothing to do with $E_\epsilon^\perp$). In this paragraph, we want to understand the oscillatory behaviour of $E_{\epsilon,\parallel}$. We will see that the dynamics in $x_\perp$ does not interfer too much with the equations on $E_{\epsilon,\parallel}$, so that we get almost the same description of oscillations as in Grenier's paper \cite{Gre1}.

First we differentiate twice with respect to time the Poisson equation satisfied by $V_\epsilon$:

\begin{equation}
\label{eq-deriv}
 \epsilon \partial^2_t \partial_{x_\parallel} E_{\epsilon,\parallel} = \partial^2_t \int \rho_\epsilon dx_\perp.
\end{equation}

We use the equation on $\rho_\epsilon$ to compute the right hand side of (\ref{eq-deriv}).

\begin{equation}
 \partial_t \int\rho_\epsilon dx_\perp = \underbrace{- \int \nabla_\perp( E^\perp_\epsilon \rho_\epsilon)dx_\perp}_{=0} - \partial_{x_\parallel} \int \rho_\epsilon v_\epsilon dx_\perp.
\end{equation}
Then we integrate with respect to $x_\perp$ the equation satisfied by $\rho_\epsilon v_\epsilon$, that is:

\[
\partial_t (\rho_\epsilon v_\epsilon) + \nabla_\perp(E_\epsilon^\perp \rho_\epsilon v_\epsilon) + \partial_{x_\parallel}(v_\epsilon^2 \rho_\epsilon) = -\rho_\epsilon(\epsilon\partial_{x_\parallel} \phi_\epsilon(t,x) +\partial_{x_\parallel} V_\epsilon(t,x_\parallel))
\]

and we get:

\begin{equation}
-\partial_t \int \rho_\epsilon v_\epsilon dx_\perp = \partial_{x_\parallel} \int \rho_\epsilon v_\epsilon^2 dx_\perp + E_{\epsilon,\parallel}\int\rho_\epsilon dx_\perp- \int \rho_\epsilon (\epsilon \partial_{x_\parallel} \phi_\epsilon) dx_\perp,
\end{equation}
so that:
$$\partial_t^2 \int \rho_\epsilon dx_\perp = \partial^2_{x_\parallel} \int \rho_\epsilon v_\epsilon^2 dx_\perp - \partial_{x_\parallel}(E_{\epsilon,\parallel} \int \rho_\epsilon dx_\perp)+ \partial_{x_\parallel}\int \rho_\epsilon (\epsilon \partial_{x_\parallel} \phi_\epsilon) dx_\perp.$$
Thus it comes:

\begin{equation}
\label{waves}
 \epsilon \partial_t^2 \partial_{x_\parallel}E_{\epsilon,\parallel} + \partial_{x_\parallel}E_{\epsilon,\parallel} = \partial^2_{x_\parallel} \int \rho_\epsilon v_\epsilon^2 dx_\perp + \epsilon \partial_{x_\parallel}[E_{\epsilon,\parallel} \partial_{x_\parallel}E_{\epsilon,\parallel} ]-\partial_{x_\parallel}\int \rho_\epsilon (\epsilon \partial_{x_\parallel} \phi_\epsilon) dx_\perp.
\end{equation}

Equation (\ref{waves}) is the wave equation allowing to describe the essential oscillations. At least formally, this equation indicates that there are time oscillations with frequency $\frac{1}{\sqrt{\epsilon}}$ and magnitude $\frac{1}{\sqrt{\epsilon}}$ created by the right-hand side of the equation which acts like a source. We observe here that the source is expected to be of order $\mathcal{O}(1)$: indeed, by assumption on the data at $t=0$, we can check that this quantity is bounded in a $B_\delta$ space. 

In particular if we want to prove strong convergence results we will have to introduce non-trivial correctors in order to get rid of these oscillations. We notice also that (\ref{waves}) is very similar to the wave equation obtained in \cite{Gre1} (the only difference is a new term in the source), so that most of the calculations and estimates on $E_{\epsilon,\parallel}$ we will need are done in \cite{Gre1}. 

\subsection{A priori estimates}

We have just observed that $E_{\epsilon,\parallel}$ roughly behaves like $\frac{1}{\sqrt{\epsilon}} e^{\pm it/\sqrt{\epsilon}}$. Hence if we consider the average in time 
\begin{equation}
 G_\epsilon = \int_0^t E_{\epsilon,\parallel}(s,x_\parallel) ds,
\end{equation}
we expect that $G_\epsilon$ is bounded uniformly with respect to $\epsilon$ in some functional space.
We also introduce the translated current (which corresponds to some filtering of the time oscillations created by the electric field):
\begin{equation}
 w_\epsilon =v_\epsilon - G_\epsilon,
\end{equation}
so that system (\ref{sys}) now writes:
\begin{equation}
 \left\{
\begin{array}{ll}
  \partial_t \rho_\epsilon + \nabla_\perp (E^\perp_\epsilon \rho_\epsilon) + \partial_\parallel((w_{\epsilon}+G_\epsilon) \rho_\epsilon)= 0 \\
  \partial_t w_{\epsilon} + \nabla_\perp (E^\perp_\epsilon (w_{\epsilon}+G_\epsilon)) + (w_{\epsilon} +G_\epsilon)\partial_\parallel(v_{\epsilon}+G_\epsilon)) = -\epsilon\partial_\parallel \phi_\epsilon(t,x_\parallel).
\end{array}
\right.
\end{equation}

The goal is now to prove some a priori estimates for $G_\epsilon, \rho_\epsilon$ and $w_\epsilon$. We are also able to get similar estimates on $E_\epsilon^\perp$ and $\epsilon \partial_{x_\parallel} \phi_\epsilon$, thanks to the Poisson equation satisfied by $\phi_\epsilon$.

\subsubsection{Estimate on $G_\epsilon$ and $\sqrt{\epsilon} E_{\epsilon,\parallel}$}

We use Duhamel's formula for the wave equation (\ref{waves}) to get the following identity:
\begin{equation}
 \mathcal{F}_\parallel G_\epsilon(t,k_\parallel) = \int_0^t \left(\frac{1}{ik_\parallel} \left[1- \cos(\frac{t-s}{\sqrt{\epsilon}})\right] \mathcal{F}_\parallel g_\epsilon(s,k_\parallel)\right) ds + \mathcal{F}_\parallel G_\epsilon^0,
\end{equation}
denoting by $\mathcal{F}_\parallel$ the Fourier transform with respect to the parallel variable only and $k_\parallel$ the Fourier variable and where:
$$g_\epsilon=\partial^2_{x_\parallel} \int \rho_\epsilon v_\epsilon^2 dx_\perp + \epsilon \partial_{x_\parallel}[E_{\epsilon,\parallel} \partial_{x_\parallel}E_{\epsilon,\parallel} ]-\partial_{x_\parallel}\int \rho_\epsilon (\epsilon \partial_{x_\parallel} \phi_\epsilon) dx_\perp,$$
\begin{equation}
 \label{G0}
G_\epsilon^0 = \sqrt{\epsilon}E_{\epsilon,\parallel}(0,x_\parallel) \sin \left(\frac{s}{\sqrt{\epsilon}}\right)-\epsilon \partial_t E_{\epsilon,\parallel}(0,x_\parallel) \left(\cos\left(\frac{s}{\sqrt{\epsilon}}\right)-1\right).
\end{equation}

We now estimate $\Vert G_\epsilon \Vert_{\delta_0}$.

\begin{eqnarray*}
 \vert G_\epsilon \vert_\delta &\leq&  \int_0^t \left\vert\mathcal{F}_\parallel^{-1}\left(\frac{1}{ik_\parallel} [1- \cos(\frac{t-s}{\sqrt{\epsilon}})] \mathcal{F}_\parallel g_\epsilon(s,k_\parallel)\right)\right\vert_\delta ds +  \vert G_\epsilon^0 \vert_\delta.
\end{eqnarray*}

$$\frac{1}{ik_\parallel}\mathcal{F}_\parallel g_\epsilon=\mathcal{F}_\parallel\left(\partial_{x_\parallel} \int \rho_\epsilon v_\epsilon^2 dx_\perp\right) + \epsilon \mathcal{F}_\parallel\left(E_{\epsilon,\parallel}\partial_{x_\parallel}E_{\epsilon,\parallel}\right).$$ 

Thanks to Remark \ref{rque} and Lemma \ref{lemma} , $(i)$:

\begin{equation}
 \left\vert  \int \partial_{x_\parallel}(\rho_\epsilon v_\epsilon^2 )dx_\perp \right\vert_\delta \leq \left\vert \partial_{x_\parallel}(\rho_\epsilon v_\epsilon^2) \right\vert_\delta \leq
(\delta_0-\delta- \frac{s}{\eta})^{-\beta}\Vert \rho_\epsilon \Vert_{\delta _0}\Vert v_\epsilon \Vert_{\delta _0}^2.
\end{equation}

Similarly, we prove:
$$\epsilon \left\vert E_{\epsilon,\parallel}\partial_{x_\parallel}E_{\epsilon,\parallel} \right\vert_\delta \leq \frac 1 2 (\delta_0-\delta- \frac{s}{\eta})^{-\beta} \Vert \sqrt{\epsilon} E_{\epsilon,\parallel}\Vert_{\delta _0}^2,$$

$$
\left\vert \int \partial_{x_\parallel}\left(\rho_\epsilon (\epsilon \partial_{x_\parallel} \phi_\epsilon)\right) dx_\perp \right\vert_\delta \leq (\delta_0-\delta- \frac{s}{\eta})^{-\beta}\Vert \rho_\epsilon \Vert_{\delta _0}\Vert \epsilon \partial_{x_\parallel} \phi_\epsilon \Vert_{\delta _0}.
$$

Thus, we have:
\begin{eqnarray*}
  \vert  G_\epsilon \vert_\delta \leq C \int_0^t (\delta_0-\delta- \frac{s}{\eta})^{(-\beta)}(\Vert \rho_\epsilon \Vert_{\delta _0} \Vert v_\epsilon \Vert_{\delta _0}^2 +  \Vert \sqrt{\epsilon} E_{\epsilon,\parallel}\Vert_{\delta _0}^2+\Vert \rho_\epsilon \Vert_{\delta _0}\Vert \epsilon \partial_{x_\parallel} \phi_\epsilon \Vert_{\delta _0}) +   \vert  G_\epsilon^0 \vert_\delta.
\end{eqnarray*}

Likewise, one can show (this time we use lemma \ref{lemma}, $(v)$):
\begin{eqnarray*}
  \vert \partial_{x_\parallel} G_\epsilon \vert_\delta \leq C \int_0^t (\delta_0-\delta- \frac{s}{\eta})^{(-\beta-1)}(\Vert \rho_\epsilon \Vert_{\delta _0} \Vert v_\epsilon \Vert_{\delta _0}^2 +  \Vert \sqrt{\epsilon} E_{\epsilon,\parallel}\Vert_{\delta _0}^2+\Vert \rho_\epsilon \Vert_{\delta _0}\Vert \epsilon \partial_{x_\parallel} \phi_\epsilon \Vert_{\delta _0}) +   \vert \partial_{x_\parallel} G_\epsilon^0 \vert_\delta.
\end{eqnarray*}
Hence using the elementary estimates
$$\int_0^t \frac{ds}{(\delta_0-\delta - \frac{\sigma}{\eta})^\beta}\leq \eta \frac{2}{1-\beta}\delta_0^{1-\beta} ,$$
$$\int_0^t \frac{ds}{(\delta_0-\delta - \frac{\sigma}{\eta})^{\beta+1}}  \leq \frac{2\eta}{\beta}(\delta_0-\delta - \frac{t}{\eta})^{-\beta}, $$
and $\Vert v_\epsilon \Vert_{\delta_0} \leq \Vert w_\epsilon \Vert_{\delta_0}  +\Vert G_\epsilon \Vert_{\delta_0}$, we get:
\begin{equation}
 \Vert G_\epsilon \Vert_{\delta_0} \leq \eta C(\delta_0,\beta) \left((\Vert w_\epsilon \Vert_{\delta_0}+ \Vert G_\epsilon \Vert_{\delta_0})^2\Vert \rho_\epsilon \Vert_{\delta_0}+ \Vert \sqrt{\epsilon}E_{\epsilon,\parallel} \Vert_{\delta_0}^2+\Vert \rho_\epsilon \Vert_{\delta _0}\Vert \epsilon \partial_{x_\parallel} \phi_\epsilon \Vert_{\delta _0}\right) + \Vert G_\epsilon^0 \Vert_{\delta_0}.
\end{equation}
If we compare two solutions $( w^{(1)}, \rho^{(1)})$ and $(w^{(2)}, \rho^{(2)})$ with the same inital data we obtain:
\begin{eqnarray}
\label{eq-comp1}
  \Vert  G^{(1)}_\epsilon - G_\epsilon^{(2)} \Vert_{\delta_0}  &\leq& \eta C\Big( (\Vert w_\epsilon^{(1)} - w^{(2)}_\epsilon \Vert_{\delta_0} + \Vert  G_\epsilon^{(1)}-  G^{(2)}_\epsilon \Vert_{\delta_0}) \nonumber \\ 
& &\times (\Vert w_\epsilon^{(1)} \Vert_{\delta_0} + \Vert w^{(2)}_\epsilon \Vert_{\delta_0}  + \Vert G_\epsilon^{(1)} \Vert_{\delta_0} + \Vert G^{(2)}_\epsilon \Vert_{\delta_0}) ( \Vert\rho^{(1)}_\epsilon \Vert_{\delta_0} + \Vert \rho^{(2)}_\epsilon \Vert_{\delta_0}) \nonumber\\
&+& (\Vert w_\epsilon^{(1)} \Vert_{\delta_0}^2 + \Vert  w^{(2)}_\epsilon \Vert_{\delta_0}^2  + \Vert G_\epsilon^{(1)} \Vert_{\delta_0}^2 + \Vert  G^{(2)}_\epsilon \Vert_{\delta_0}^2)(\Vert \rho_\epsilon^{(1)} - \rho^{(2)}_\epsilon \Vert_{\delta_0}) \nonumber \\
&+& \Vert \rho_\epsilon^{(1)}-\rho_\epsilon^{(2)} \Vert_{\delta_0}(\Vert\epsilon \partial_{x_\parallel}\phi_\epsilon^{(1)} \Vert_{\delta_0}+ \Vert\epsilon \partial_{x_\parallel}\phi_\epsilon^{(2)} \Vert_{\delta_0})\nonumber \\
&+& (\Vert \rho_\epsilon^{(1)} \Vert_{\delta_0}+\Vert \rho_\epsilon^{(2)} \Vert_{\delta_0})\Vert \epsilon \partial_{x_\parallel}\phi_\epsilon^{(1)}-\epsilon \partial_{x_\parallel}\phi_\epsilon^{(2)}\Vert_{\delta_0} \nonumber \\
&+& \Vert \sqrt{\epsilon} E^{(1)}_{\epsilon,\parallel} - \sqrt{\epsilon}E^{(2)}_{\epsilon,\parallel} \Vert_{\delta_0}(\Vert \sqrt{\epsilon} E^{(1)}_{\epsilon,\parallel}\Vert_{\delta_0} + \Vert \sqrt{\epsilon} E^{(2)}_{\epsilon,\parallel}\Vert_{\delta_0})\Big) .
\end{eqnarray}

Likewise we get the same estimates on $\Vert\sqrt{\epsilon}E_{\epsilon,\parallel}\Vert_{\delta_0}$ since we have the formula:
\begin{equation}
 \mathcal{F}_\parallel (\sqrt{\epsilon} E_{\epsilon,\parallel})(t,k_\parallel) = \int_0^t \left(\frac{1}{ik_\parallel} [\sin(\frac{t-s}{\sqrt{\epsilon}})] \mathcal{F}_\parallel g_\epsilon(s,k_\parallel)\right) ds + \mathcal{F}_\parallel (\sqrt{\epsilon} E_{\epsilon,\parallel}^0),
\end{equation}
with
\begin{equation}
 \label{E0}
E_{\epsilon,\parallel}^0 = E_{\epsilon,\parallel}(0,x) \cos(\frac{s}{\sqrt{\epsilon}})+\sqrt{\epsilon} \partial_t E_{\epsilon,\parallel}(0,x) \sin(\frac{s}{\sqrt{\epsilon}}).
\end{equation}

\subsubsection{Estimate on $E_\epsilon^\perp$ and $\epsilon \partial_{x_\parallel} \phi_\epsilon$}

We now use the scaled Poisson equation satisfied by $\phi_\epsilon$ to get some a priori estimates.

The principle here is to look at the symbols of the operators involved in the  Poisson equations. Accordingly, we compute in Fourier variables:
\begin{equation}
\epsilon^2k_\parallel^2 \mathcal{F}\phi_\epsilon + \vert k_\perp\vert^2 \mathcal{F}\phi_\epsilon= \mathcal{F}\left(\rho_\epsilon - \int \rho_\epsilon dx_\perp\right).
\end{equation}
Thus it comes:

$$\mathcal{F}\phi_\epsilon=\frac{\mathcal{F}(\rho_\epsilon - \int \rho_\epsilon dx_\perp)}{\epsilon^2 k_\parallel^2 + \vert k_\perp\vert^2}.$$

Since $\int (\rho_\epsilon - \int \rho_\epsilon dx_\perp) dx_\perp =0$, we have for all $k_\parallel$:
$$\mathcal{F}\left(\rho_\epsilon - \int \rho_\epsilon dx_\perp\right)(0,k_\parallel)=0.$$

Thus it comes, for all $k_\perp, k_\parallel$:
\begin{equation}
 \vert \mathcal{F}\phi_\epsilon \vert \leq \frac{\vert \mathcal{F}(\rho_\epsilon - \int \rho_\epsilon dx_\perp)\vert}{\vert k_\perp\vert^2}.
\end{equation}
In particular we easily get, using the relation $E_\epsilon^\perp = -\nabla^\perp \phi_\epsilon$:

\begin{equation*}
 \vert \mathcal{F}E^\perp_\epsilon \vert \leq \frac{\vert \mathcal{F}(\rho_\epsilon - \int \rho_\epsilon dx_\perp)\vert}{\vert k_\perp\vert} \leq \left\vert \mathcal{F}\left(\rho_\epsilon - \int \rho_\epsilon dx_\perp\right)\right\vert.
\end{equation*}

Hence:
\begin{equation}
\label{elec}
 \Vert E^\perp_\epsilon \Vert_{\delta_0} \leq C \Vert\rho_\epsilon \Vert_{\delta_0}.
\end{equation}

Likewise, since $a b \leq \frac{1}{2}(a^2 + b^2)$ and $\vert k_\perp\vert \geq 1$:

\begin{equation*}
 \vert \mathcal{F}(\epsilon \partial_{x_\parallel} \phi_\epsilon )\vert \leq \frac{\epsilon \vert k_\parallel \vert \vert \mathcal{F}(\rho_\epsilon - \int \rho_\epsilon dx_\perp)\vert}{\epsilon^2 k_\parallel^2 + \vert k_\perp\vert^2} \leq \frac{1}{2} \vert \mathcal{F}(\rho_\epsilon - \int \rho_\epsilon dx_\perp)\vert,
\end{equation*}

and consequently:
\begin{equation}
 \Vert \epsilon \partial_{x_\parallel} \phi_\epsilon \Vert_{\delta_0} \leq C \Vert\rho_\epsilon \Vert_{\delta_0}.
\end{equation}

And if we compare two solutions with the same initial data:
\begin{equation}
 \Vert E^{\perp,(1)}_\epsilon  - E^{\perp,(2)}_\epsilon \Vert_{\delta_0} +  \Vert \epsilon \partial_{x_\parallel}  \phi^{(1)}_\epsilon - \epsilon\partial_{x_\parallel}  \phi^{(2)}_\epsilon  \Vert_{\delta_0}  \leq C \Vert \rho^{(1)}_\epsilon - \rho^{(2)}_\epsilon\Vert_{\delta_0}.
\end{equation}

\subsubsection{Estimate on $\rho_\epsilon$ and $w_\epsilon$}
We now use the conservation laws satisfied by $\rho_\epsilon$ and $w_\epsilon$ to get the appropriate estimates.

The density $\rho_\epsilon$ satisfies the equation:
$$  \partial_t \rho_\epsilon + \nabla_\perp (E^\perp_\epsilon \rho_\epsilon) + \partial_\parallel((w_{\epsilon}+G_\epsilon) \rho_\epsilon)= 0.$$
With the same kind of computations as before and thanks to estimate (\ref{elec}) we get:
\begin{eqnarray*}
 \vert \rho_\epsilon \vert_\delta &\leq& \int_0^t \vert \partial_t \rho_\epsilon \vert_\delta + \vert \rho_\epsilon(0) \vert_\delta \\
&\leq& \Vert \rho_\epsilon(0) \Vert_{\delta_0} + \int_0^t (\delta_0-\delta-\frac{s}{\eta})^{-\beta} \Vert \rho_\epsilon \Vert_{\delta_0}(\Vert \rho_\epsilon \Vert_{\delta_0}+\Vert w_\epsilon \Vert_{\delta_0}+\Vert G_\epsilon \Vert_{\delta_0})ds .
\end{eqnarray*}

Similarly we estimate $\vert \partial_{x_i} \rho_\epsilon\vert_{\delta}$ by differentiating with respect to $x_i$ the equation satisfied by $\rho_\epsilon$.
Finally we get:
\begin{equation}
\label{reg-rho}
 \Vert \rho_\epsilon \Vert_{\delta_0} \leq \eta C \Vert \rho_\epsilon \Vert_{\delta_0}(\Vert \rho_\epsilon \Vert_{\delta_0}+\Vert w_\epsilon \Vert_{\delta_0}+\Vert G_\epsilon \Vert_{\delta_0}).
\end{equation}

If we compare two solutions with the same initial conditions, we get likewise:
\begin{eqnarray}
\label{eq-comp2}
 \Vert \rho^{(1)}_\epsilon -\rho^{(2)}_\epsilon \Vert_{\delta_0}  &\leq& \eta C \Big( (\Vert \rho^{(1)}_\epsilon \Vert_{\delta_0} + \Vert \rho^{(2)}_\epsilon \Vert_{\delta_0})(\Vert w^{(1)}_\epsilon - w^{(2)}_\epsilon \Vert_{\delta_0}+\Vert  G^{(1)}_\epsilon - G^{(2)}_\epsilon \Vert_{\delta_0}) \nonumber \\
&+& (\Vert \rho^{(1)}_\epsilon \Vert_{\delta_0} + \Vert \rho^{(2)}_\epsilon \Vert_{\delta_0} +\Vert w^{(1)}_\epsilon \Vert_{\delta_0}+\Vert w^{(2)}_\epsilon \Vert_{\delta_0}+\Vert G^{(1)}_\epsilon \Vert_{\delta_0}+\Vert G^{(2)}_\epsilon \Vert_{\delta_0})\nonumber \\
&\times& (\Vert \rho^{(1)}_\epsilon -\rho^{(2)}_\epsilon \Vert_{\delta_0})\Big).
\end{eqnarray}

In the same fashion, we estimate the $\delta_0$ norm of $w_\epsilon$:
\begin{equation}
 \Vert w_\epsilon \Vert_{\delta_0} \leq \eta C\left((\Vert w_\epsilon \Vert_{\delta_0}+1)\Vert \rho_\epsilon \Vert_{\delta_0} + (\Vert w_\epsilon \Vert_{\delta_0}+\Vert G_\epsilon \Vert_{\delta_0})^2+\Vert \epsilon \partial_\parallel \phi_\epsilon\Vert_{\delta_0}  \right),
\end{equation}
and if we compare two solutions with the same initial data:
\begin{eqnarray}
\label{eq-comp3}
 \Vert w^{(1)}_\epsilon -w^{(2)}_\epsilon \Vert_{\delta_0}  &\leq& \eta C \Big( (\Vert \rho^{(1)}_\epsilon \Vert_{\delta_0}+ \Vert \rho^{(2)}_\epsilon \Vert_{\delta_0})(\Vert w^{(1)}_\epsilon \Vert_{\delta_0} + \Vert w^{(2)}_\epsilon \Vert_{\delta_0}) \nonumber \\
&+& (\Vert  w^{(1)}_\epsilon -   w^{(2)}_\epsilon \Vert_{\delta_0}+ \Vert \rho^{(1)}_\epsilon -  \rho^{(2)}_\epsilon \Vert_{\delta_0}) \nonumber \\
&+& (\Vert w^{(1)}_\epsilon \Vert_{\delta_0}+\Vert w^{(2)}_\epsilon \Vert_{\delta_0}+\Vert G^{(1)}_\epsilon \Vert_{\delta_0}+\Vert G^{(2)}_\epsilon \Vert_{\delta_0})\nonumber \\
& & \times(\Vert w^{(1)}_\epsilon -  w^{(2)}_\epsilon \Vert_{\delta_0}+\Vert G^{(1)}_\epsilon - G^{(2)}_\epsilon \Vert_{\delta_0}) \nonumber \\
& &\Vert \epsilon \partial_\parallel \phi^{(1)}_\epsilon -  \epsilon \partial_\parallel \phi^{(1)}_\epsilon\Vert_{\delta_0} \Big).
\end{eqnarray}

\subsection{Approximation scheme} 
\label{sec-approx}
We use the usual approximation scheme for Cauchy-Kovalevskaya type of results (\cite{Caf}).

We define $\rho_{\epsilon}^{n},w_{\epsilon}^{n},G_{\epsilon}^{n},V_{\epsilon}^{n}, \phi_{\epsilon}^{n}$ by recursion:

\underline{Initialization}
For $0<t<\eta(\delta_0 -1)$, we define:

$$ \rho_\epsilon^0(t)=\rho_\epsilon(0), $$
$$ w_{\epsilon}^{0}=v_\epsilon(0) - G_\epsilon^0,$$
$$  -\epsilon^2 \partial^2_{x_\parallel} \phi^{0}_\epsilon - \Delta_{x_\perp} \phi^{0}_\epsilon = \rho^{0}_\epsilon - \int \rho^{0}_\epsilon dx_\perp,$$

$$E_{\epsilon}^{\perp, 0}= - \nabla^\perp \phi_\epsilon^0,$$

and $-\epsilon^2 \partial_{x_\parallel}^2 V_{\epsilon}^{0}=\rho_\epsilon^0 - \int \rho_\epsilon^0 dx$. Finally, $G_\epsilon^0$ is given by formula $G_\epsilon^0= -\int_0^t \partial_\parallel V_\epsilon^{0}$.

\underline{Recursion} We define $\rho_\epsilon^{n+1}, w_\epsilon^{n+1}$ by the equations:
$$ \left\{
\begin{array}{ll}
  \partial_t \rho_\epsilon^{n+1} + \nabla_\perp (E^{\perp,n}_\epsilon \rho^n_\epsilon) + \partial_\parallel((w^n_{\epsilon}+G^n_\epsilon) \rho^n_\epsilon)= 0 \\
  \partial_t w_{\epsilon}^{n+1} + \nabla_\perp (E^{\perp,n}_\epsilon (w^n_{\epsilon}+G^n_\epsilon)) + (w^n_{\epsilon} +G^n_\epsilon)\partial_\parallel(v^n_{\epsilon}+G^n_\epsilon)) = -\epsilon\partial_\parallel \phi^n_\epsilon(t,x_\parallel).
\end{array}
\right.$$

with the initial conditions: $\rho_\epsilon^{n+1}(0)=\rho_\epsilon(0)$ and $ w_{\epsilon}^{n+1}=v_\epsilon(0) - G_\epsilon^0$.

Then we can define $\phi_\epsilon^{n+1}$ as the solution to the Poisson equation:
$$-\epsilon^2 \partial^2_{x_\parallel} \phi^{n+1}_\epsilon - \Delta_{x_\perp} \phi^{n+1}_\epsilon = \rho^{n+1}_\epsilon - \int \rho^{n+1}_\epsilon dx_\perp.$$

$$E_{\epsilon}^{\perp, n+1}= - \nabla^\perp \phi_\epsilon^{n+1},$$

Similarly,
$$-\epsilon^2 \partial_{x_\parallel}^2 V_{\epsilon}^{n+1}=\rho_\epsilon^{n+1}- \int \rho_\epsilon^{n+1} dx.$$
Then we can define $G_\epsilon^{n+1}$ with the formula: $G_\epsilon^{n+1}= -\int_0^t \partial_\parallel V_\epsilon^{n+1}$.

Now let $C_1$ be a constant larger than $\Vert \rho_\epsilon(0) \Vert_{\delta_0}$, $\Vert w_\epsilon(0) \Vert_{\delta_0}$, $\Vert G_\epsilon(0) \Vert_{\delta_0}$, $\Vert \sqrt{\epsilon} E_\epsilon(0) \Vert_{\delta_0}$ and all the other constants in the previous estimates. It is possible to choose $\eta$ small enough with respect to $C_1$ to propagate the following estimates by recursion (we refer to \cite{Gre1} for details; we use in particular estimates (\ref{eq-comp1}),(\ref{eq-comp2}),(\ref{eq-comp3})). There exists $C_2>C_1$, for all $n\geq 1$:
\begin{enumerate}[(i)]
\item

$$ \left\{
\begin{array}{ll}
\Vert \rho_\epsilon^n  \Vert_{\delta_0}\leq {C_2},   \\
\Vert w_\epsilon^n \Vert_{\delta_0}\leq \ C_2,  \\ 
\Vert G_\epsilon^n \Vert_{\delta_0}\leq C_2,   \\
\Vert \sqrt{\epsilon} E_{\epsilon,\parallel}^n \Vert_{\delta_0}\leq C_2.
\end{array}
\right.$$

\item

$$ \left\{
\begin{array}{ll}
\Vert \rho_\epsilon^n - \rho_\epsilon^{n-1} \Vert_{\delta_0}\leq \frac{C_2}{2^n},   \\
 \Vert w_\epsilon^n - w_\epsilon^{n-1} \Vert_{\delta_0}\leq \frac{C_2}{2^n}, \\ 
 \Vert G_\epsilon^n - G_\epsilon^{n-1} \Vert_{\delta_0}\leq \frac{C_2}{2^n}, \\
 \Vert \sqrt{\epsilon} E_{\epsilon, \parallel}^n -\sqrt{\epsilon} E_{\epsilon, \parallel}^{n-1} \Vert_{\delta_0}\leq \frac{C_2}{2^n}.
\end{array}
\right.$$

\end{enumerate}

This proves that the sequences $\rho_{\epsilon}^{n},w_{\epsilon}^{n},G_{\epsilon}^{n}, \sqrt{\epsilon}E_\epsilon, E_{\epsilon}^{\perp, n},\epsilon\partial_{x_\parallel}\phi_{\epsilon}^{n}$ are Cauchy sequences (with respect to $n$) in $B_{\delta_0}^\eta$, and consequently converge strongly in $B_{\delta_0}^\eta$, the estimates being uniform in $\epsilon$.
It is clear that the limit satisfies System (\ref{sys}).

The requirement $\delta_1<\delta_0$ and the explicit life span in Theorem \ref{exi} come directly from the definition of the $B_{\delta_0}^\eta$ spaces.

For the uniqueness part, one can simply notice that the estimates we have shown allow us to prove that the application $F$ defined by:
$$ F(\rho_\epsilon, w_\epsilon ) =\begin{pmatrix} \int_0^t ( - \nabla_\perp (E^{\perp}_\epsilon \rho_\epsilon) - \partial_\parallel((w_{\epsilon}+G_\epsilon) \rho_\epsilon)) )ds \\ \int_0^t(-\nabla_\perp (E^{\perp}_\epsilon (w_{\epsilon}+G_\epsilon)) - (w_{\epsilon} +G_\epsilon)\partial_\parallel(v_{\epsilon}+G_\epsilon))  -\epsilon\partial_\parallel \phi_\epsilon(t,x_\parallel))ds \end{pmatrix}, $$
is a contraction on the closed subset $B$ of $B_{\delta_0}\times B_{\delta_0}$, defined by:
\[
B= \left\{\rho, w \in B_{\delta_0}; \Vert \rho\Vert_{{\delta_0}}\leq C,  \Vert w\Vert_{{\delta_0}}\leq C \right\},
\]
with $C$ large enough, provided that $\eta$ is chosen small enough. The uniqueness of the analytic solution then follows.

\subsection*{Proof of Proposition \ref{exi2}}

We can lead the same analysis as for the proof of Theorem \ref{exi2}, but even simpler since here we do not have to deal anymore with the fast oscillations in time.
The only slightly different point is to estimate the norm of $\int_0^t -\partial_\parallel p ds = \int_0^t \partial_\parallel \int \rho v^2 dx_\perp ds$, which is straightforward:

$$\left\Vert \int_0^t  \partial_\parallel p ds\right\Vert _{{\delta_0}}  \leq \eta C \Vert \rho \Vert _{{\delta_0}}\Vert v \Vert _{{\delta_0}}^2.$$

Then as before, we can use a contraction argument to prove the proposition.

\section{Proof of Theorem \ref{con}}
\label{sec-proof2}
\subsection*{Step 1: Another average in time for $E_{\epsilon,\parallel}$}

We have observed previously that the wave equation (\ref{waves}) describing the time oscillations of $E_{\epsilon,\parallel}$ was the same as the one appearing in Grenier's work, except for a slight change in the source. Therefore the following decomposition taken from \cite{Gre1} identically holds:

\begin{lem}
\label{osci}
 Under assumption $(H)$, there exist vector fields $E_\epsilon^1, E_\epsilon^2$ and $W_\epsilon$ such that $E_{\epsilon,\parallel}=E_\epsilon^1+E_\epsilon^2$ and a positive constant $C$ independent of $\epsilon$ such as:
\begin{enumerate}[(i)]
 \item $\Vert \sqrt{\epsilon}E_\epsilon^1 \Vert_{L^\infty(H^{s-1}_{x_\parallel})}\leq C$.
 \item $\partial_t W_\epsilon = E_\epsilon^1$, $\Vert W_\epsilon \Vert_{L^\infty(H^{s-1}_{x_\parallel})}\leq C$ and $W_\epsilon \rightharpoonup 0$ in $L^2$.
 \item $W^\epsilon(0)= -\epsilon \partial_t E_{\epsilon,\parallel}(0) = \int \rho_\epsilon(0)v_\epsilon(0) dx_\perp$. 
 \item $\Vert E_\epsilon^2 \Vert_{L^\infty(H^{s-1}_{x_\parallel})}\leq C$.
 \item $\int E^1_\epsilon dx_\parallel =  \int E^2_\epsilon dx_\parallel=0$.
\end{enumerate}

\end{lem}

\begin{proof}[Idea of the proof] The idea in order to build $E^2_\epsilon$ is to cut off the essential temporal oscillations (of frequency $\frac{1}{\sqrt{\epsilon}}$). Hence, we can define $E^2_\epsilon$ defined by its Fourier transform:
$$\mathcal{F}_\parallel E^2_\epsilon (t,k_\parallel)= \frac{1}{2\pi \sqrt{\epsilon}} \int_t^{t+2\pi \sqrt{\epsilon}} \mathcal{F}_\parallel  E_{\epsilon,\parallel}(s,k_\parallel) ds$$ and $E^1_\epsilon=E_{\epsilon,\parallel} - E^2_\epsilon$, so that $E^1_\epsilon$ gathers the essential information on the oscillations.

Let us refer to \cite{Gre1} for details.
 
\end{proof}

\subsection*{Step 2: Uniform bound on $E_\epsilon^\perp$ and $\partial_{x_\parallel} \phi_\epsilon$}

Under hypothesis $(H)$ we clearly get that $E_\epsilon^\perp$  and $\partial_{x_\parallel} \phi_\epsilon$ are bounded in $L^\infty_t(H^{s-1})$ uniformly with respect to $\epsilon$ (we do not need any gain of elliptic regularity).

Since 
$$\int(\rho_\epsilon - \int \rho_\epsilon dx_\perp)dx_\perp=0,$$ 
we easily check that:
$$\Vert \phi_\epsilon\Vert_{H^s_{x_\perp,x_\parallel}} \leq \Vert \rho -\int \rho dx_\perp \Vert_{H^s_{x_\perp,x_\parallel}}.$$
Hence the result.

\subsection*{Step 3: Passage to the strong limit}
Let $w_\epsilon=v_\epsilon - W_\epsilon$.
According to Lemma \ref{osci}, $w_\epsilon$ is uniformly bounded in $L^\infty_t(H^{s-1})$. On the other hand, we have :
\begin{equation}
 \partial_t w_\epsilon + \nabla_\perp(E_\epsilon^\perp w_\epsilon) + w_\epsilon \partial_{x_\parallel} w_\epsilon =-\epsilon \partial_{x_\parallel}\phi_\epsilon + E_\epsilon^2  - w_\epsilon \partial_{x_\parallel} W_\epsilon - W_\epsilon \partial_{x_\parallel} w_\epsilon -W_\epsilon \partial_{x_\parallel} W_\epsilon.
\end{equation}
(Notice that $\nabla_{\perp}(E^\perp W_\epsilon)= W_\epsilon \nabla_{\perp}(E^\perp)=0$.)

Thus, using the uniform bounds, we can see that $\partial_t w_\epsilon$ is bounded in $L^\infty_t(H^{s-2})$ and thanks to the Aubin-Lions lemma (see J. Simon \cite{sim}), $w_\epsilon$ converges strongly (up to a subsequence) to some function $w$ in $\mathcal{C}([0,T], H^{s'-1})$ for all $s'<s$. 

According to Step 2, $\epsilon \partial_{x_\parallel} \phi_\epsilon \rightharpoonup 0$ in the distributional sense. The following convergence also holds in the sense of distributions, according to Lemma \ref{osci}:
$$w_\epsilon \partial_{x_\parallel} W_\epsilon +  W_\epsilon \partial_{x_\parallel} w_\epsilon \rightharpoonup 0,$$
and $W_\epsilon \partial_{x_\parallel} W_\epsilon+E_\epsilon^2$ weakly converges to some function  $F$ since it is clearly bounded in $L^\infty(H^{s-2}_{x_\parallel})$.

Furthermore, since:
\[
\int \left(W_\epsilon \partial_{x_\parallel} W_\epsilon+E_\epsilon^2\right) dx_\parallel =0,
\]
this implies that $\int F dx_\parallel=0$, and thus there exists $p$ such that $F=-\partial_{x_\parallel} p$.

Since $E_\epsilon^\perp$ is uniformly bounded in $L^\infty_t(H^{s-1})$, it also weakly-* converges, up to a subsequence,  to some function $E^\perp$.

We now use the strong limit of $w_\epsilon$ in $\mathcal{C}([0,T], H^{s'-1})$  in order to pass to the limit in the sense of distributions in the convection terms.
As a consequence, it comes, passing to the limit in the sense of distributions:
\begin{equation}
 \partial_t w + \nabla_\perp(E^\perp w) + w \partial_{x_\parallel} w = -\partial_{x_\parallel}p.
\end{equation}

The equation satisfied by $\rho_\epsilon$ is:
$$\partial_t \rho_\epsilon + \nabla_\perp(E_\epsilon^\perp \rho_\epsilon) + \partial_\parallel (w_\epsilon \rho_\epsilon) = - \partial_\parallel (W_\epsilon \rho_\epsilon).$$

The proof is similar for $\rho_\epsilon$ which converges strongly, up to a subsequence, to $\rho$ in $\mathcal{C}([0,T], H^{s'})$ for all $s'<s$. One can likewise take limits in the Poisson equations.
We finally obtain (\ref{simple}). By uniqueness of the solutions to (\ref{simple}), the limits actually hold without extraction.

\subsection*{Step 4: Equations for the correctors}

The final step relies on the following lemma proved in Grenier's paper \cite{Gre1} (the main point is to notice that the application $f \mapsto e^{\pm it/\sqrt{\epsilon}} f$ is an isometry on $L^\infty(H^s)$ for any $s$.)
\begin{lem}
\label{osci2}
There exist two correctors $E_+(t,x_\parallel)$ and $E_-(t,x_\parallel)$ in $\mathcal{C}(H^{s-1})$ such that, for all $s'<s$:
\begin{itemize}
\item $\Vert \sqrt{\epsilon} E^1_\epsilon  - e^{it/\sqrt{\epsilon}}E_+ - e^{-it/\sqrt{\epsilon}}E_- \Vert_{\mathcal{C}(H^{s'-1})} \rightarrow 0$,
\item $\Vert W_\epsilon - \frac{1}{i}\left(e^{it/\sqrt{\epsilon}}E_+ -e^{-it/\sqrt{\epsilon}}E_-  \right)\Vert_{\mathcal{C}(H^{s'-1})} \rightarrow 0$.
\end{itemize}
\end{lem}
In particular we can deduce that:
\[
e^{-it/\sqrt{\epsilon}} \sqrt{\epsilon} E^1_\epsilon \rightharpoonup E_+
\]
(and similarly $e^{it/\sqrt{\epsilon}} \sqrt{\epsilon} E^1_\epsilon \rightharpoonup E_-$).

Then, the idea is to use Lemmas \ref{osci} and \ref{osci2} and the wave equation (\ref{waves}) in order to obtain the equations satisfied by $E_\pm$. By elementary (but rather tedious) computations we get:
$$\partial_t (\partial_{x_\parallel} E_\pm) + \left(\int \rho v dx_\perp\right) \partial_{x_\parallel} (\partial_{x_\parallel} E_\pm) =0.$$
Lemma \ref{osci2} also provides the initial conditions for $E_\pm$.

The proof of the theorem is now complete.

\section{Discussion on the sharpness of the results}
\label{sec-sharp}
\subsection{On the analytic regularity}
\label{sec-analytic1}

Let us recall that the multi-fluid system (\ref{grenierS}) is ill-posed in Sobolev spaces, because of the two-stream instabilities (remind that this is due to the coupling between the different phases of the fluid).

For system (\ref{simple}), we expect the situation to be similar. Due to the dependence on $x_\perp$ and the constraint $\int \rho dx_\perp =1$, system (\ref{simple}) is by nature a multi-fluid system.
Neverthless, one could maybe imagine that the dynamics in the $x_\perp$ variable could yield some mixing in $x_\perp$ and $x_\parallel$ (in the spirit of hypoellipticity results) and thus could perhaps  bring stability. Here we explain why this is not the case.

 The idea is to consider for (\ref{simple}) shear flows initial data. This will allow to exactly recover the multi-fluid equations (\ref{grenierS}). We take:
\[
E_0^\perp = (0, \varphi(x_1,x_\parallel),0),
\]
and consequently $\rho_0 =\nabla_\perp \wedge E_0^\perp = -\varphi'(x_1,u)$. We also assume that $v_0(x_1, x_\parallel)$ does not depend on $x_2$.

Then we observe that:
\begin{equation*}
\begin{split}
\nabla_\perp (E^\perp_0 \rho_0) = 0, \\
\nabla_\perp (E^\perp_0  v_0) = 0. 
\end{split}
\end{equation*}

With such initial data, system (\ref{simple}) reduces to:
\begin{equation}
\label{eq-multi}
\left\{
    \begin{array}{ll}
  \partial_t \rho +  \partial_\parallel(v_\parallel \rho)= 0 \\
  \partial_t v_\parallel + v_\parallel \partial_\parallel(v_\parallel) = -\partial_\parallel p(t,x_\parallel) \\
  \int \rho dx_1 = 1,\\
\end{array}
  \right.
\end{equation}
and we observe that there is no more dynamics in the $x_\perp$ variable. This is nothing but system (\ref{grenierS}) in dimension $1$, with $M=[0,1[$ and $\mu$ the Lebesgue measure.

Now, let us consider measure type of data in the $x_1$ variable for $\rho$ and $v$ (this corresponds to a ``degenerate'' version of the shear flows defined above). In particular if we choose:
\[
\varphi= \frac 1 2 \mathbbm{1}_{x_1\leq \frac 1 4} \rho_{0,1}(x_\parallel) +  \frac 1 2  \mathbbm{1}_{x_1\leq \frac 1 2} \rho_{0,2}(x_\parallel), 
\]
we get:
\begin{equation}
\begin{split}
\rho_0= \frac 1 2 \delta_{x_1=\frac 1 4} \rho_{0,1}(x_\parallel) + \frac 1 2 \delta_{x_2=\frac 1 2} \rho_{0,2}(x_\parallel), \\
v_0=\frac 1 2 \delta_{x_1=\frac 1 4} v_{0,1}(x_\parallel) + \frac 1 2 \delta_{x_1=\frac 1 2} v_{0,2}(x_\parallel) 
\end{split}
\end{equation}
and we obtain the following system for $\alpha=1, 2$:

\begin{equation}
\left\{
    \begin{array}{ll}
  \partial_t \rho_\alpha +  \partial_\parallel(v_\alpha \rho_\alpha)= 0 \\
  \partial_t v_\alpha + v_\alpha \partial_\parallel(v_\alpha) = -\partial_\parallel p(t,x_\parallel) \\
  \rho_1 + \rho_2=1.\\
\end{array}
  \right.
\end{equation}
This particular system was given as an example by Brenier in \cite{Br2} to illustrate ill-posedness in Sobolev spaces of the multi-fluid equations.

We denote $q=\rho_1 v_1$. Using the constraint $\rho^1 + \rho^2 =1$, we easily obtain that 
$$p_\parallel = -q^2 \left(\frac{1}{\rho_1}+ \frac{1}{1-\rho_1} \right).$$

We can then observe that the system:
\begin{equation}
\left\{
    \begin{array}{ll}
  \partial_t \rho_1 +  \partial_\parallel q= 0 \\
  \partial_t q + \partial_\parallel(\frac{q^2}{\rho_1}) = -\rho_1 \partial_\parallel p(t,x_\parallel) \\
\end{array}
  \right.
\end{equation}
is elliptic in space-time, and consequently it is ill-posed in Sobolev spaces.

Actually this example is not completely satisfying, since it is singular in $x_1$. Nevertheless we can consider the convolution of this initial data with a standard mollifier, which yields the same qualitative behaviour.

\subsection{On the analytic regularity in the perpendicular variable}
\label{sec-analytic2}

We observe that if the initial datum $(\rho(0),v(0))$ does not depend on $x_\parallel$, then the fluid system (\ref{simple}) reduces to:
\begin{equation}
\left\{
    \begin{array}{ll}
  \partial_t \rho + \nabla_\perp (E^\perp \rho) = 0 \\
  \partial_t v_\parallel + \nabla_\perp (E^\perp v_\parallel) =0 \\
  E^\perp= \nabla^\perp \Delta_\perp^{-1} \left(\rho - \int \rho dx_\perp\right)\\
  \int \rho dx_\perp = 1.\\
\end{array}
  \right.
\end{equation}

Thus, $\rho$ satisfies $2D$ incompressible Euler system, written in vorticity formulation.  This systems admits a unique global strong solution provided that $\rho(0) \in H^s(\mathbb{T}^2)$ (with $s>1$), by a classical result of  Kato \cite{Kat} and even a unique global weak solution provided that $\rho(0) \in L^\infty(\mathbb{T}^2)$, by a classical result of Yudovic \cite{Yud}.

In the other hand, $v_\parallel$ satisfied a transport equation with the force field $E^\perp$. If we assume for instance that $v_0$ is a bounded measure, then using the classical log-Lipschitz estimate on $E^\perp$, we get a unique global weak solution $v_\parallel$ by the method of characteristics. 

One could think that it should be possible to build solutions to the final fluid system (\ref{simple}) with similar ``weak'' regularity in the $x_\perp$ variable (while keeping analyticity in the $x_\parallel$ variable).
Actually this is not possible in general: this is related to the fact that $E^\perp$ depends also on $x_\parallel$ and this entails that we also need analytic regularity in the $x_\perp$ variable to get analytic regularity in the $x_\parallel$ variable (see estimations such as (\ref{reg-rho})).

\subsection{On the local in time existence}
\label{sec-local}

In \cite{Br3}, Brenier considers potential velocity fields,  that is velocity fields of the form $v_\Theta = \nabla_x \Phi_\Theta$, for the multi-fluid system:
\begin{equation}
\left\{
    \begin{array}{ll}
    \Theta=1,...,M \quad M\in \mathbb{N}^* \\
  \partial_t \rho_\Theta  + \operatorname{div}( \rho_\Theta v_\Theta)= 0 \\
  \partial_t v_\Theta +  v_\Theta.\nabla(v_\Theta) = E \\
  \operatorname{rot} E = 0 \\
  \sum_{\Theta=1}^M \rho_\Theta  = 1.\\
\end{array}
  \right.
\end{equation}
In this case the equation on the velocities becomes:
\begin{equation}
\partial_t \Phi_\Theta + \frac 1 2 \vert \nabla_x  \Phi_\Theta \vert^2 +p =0.
\end{equation}
It is proved in \cite{Br3} that any strong solution satisfying 
\[
\inf_{\Theta, t, x} \rho_\Theta(t,x) >0
\] 
can not be global in time unless the initial energy vanishes:
\begin{equation}
\sum_{\Theta=1}^M \int \rho_{\Theta, t=0} \vert u_{\Theta, t=0} \vert^2 dx =0.
\end{equation}
This striking result relies on a variational interpretation of these Euler equations.
Using the same particular initial data as in section \ref{sec-analytic1}, this indicates that for system (\ref{simple}) also, there is no global strong solution, unless there is no dependence on $x_\perp$ or $x_\parallel$.

We observe that if the initial datum $(\rho(0),v(0))$ does not depend on $x_\perp$, the fluid system (\ref{simple}) does not make sense anymore (as for incompressible Euler in dimension $1$). When the initial datum $(\rho(0),v(0))$ does not depend on $x_\parallel$, we have seen that we recover $2D$ incompressible Euler and there is indeed global existence (of strong or weak solutions).

\subsection{The relative entropy method applied to a toy model : failure of the multi-current limit}
\label{sec-rela}

It seems very appealing to try to use the relative entropy method (which was introduced by Brenier \cite{Br2} for Vlasov type of systems) to study the limit, as it would open the way to the study of the limit for solutions to the initial system (\ref{kinbegin}) with low regularity. The only requirement would be that the two first moments of the initial data for (\ref{kinbegin}) are in a small neighborhood (say  in $L^2$ topology) of the smooth initial data for the limit system (\ref{simple}). 
Nevertheless it is not possible to overcome the two-stream instabilities in this framework. We intend to show why.

Let us consider the toy model:
\begin{equation}
\label{toy}
\left\{
    \begin{array}{ll}
  \partial_t f_\epsilon^\theta + v.\nabla_x f_\epsilon^\theta +E_\epsilon.\nabla_v f^\theta_\epsilon=0\\
  E_\epsilon=-\nabla_x V_\epsilon\\
  -\epsilon\Delta_x V_\epsilon = \int \int f_\epsilon^\theta dv d\mu -1\\
  f^\theta_{\epsilon}(t=0)= f^\theta_{\epsilon,0}.
\end{array}
  \right.
\end{equation}
with $t>0$, $x \in \mathbb{T}^3$, $v \in \mathbb{R}^3$ and where $\theta$ lies in $[0,1]$ equipped with a positive measure $\mu$ which is:
\begin{itemize}

\item either a sum of Dirac masses with total mass $1$, such as:
\[
\mu=\sum_{i=0}^{N-1} \frac{1}{N} \delta_{\theta= i/N}.
\]
In this case, we model a plasma made of $N$ phases. 

\item or the Lebesgue measure, in which case we model a continuum of phases.

\end{itemize}

Actually, we could have considered more general Borel measures but we restrict to these cases for simplicity.
This system can be seen as the kinetic counterpart of a simplified version of (\ref{sys}), which focuses on the unstable feature of the system. Of course we could have considered directly the fluid version, that is: 
\begin{equation}
\left\{
    \begin{array}{ll}
  \partial_t \rho_\epsilon^\theta + \nabla_x (\rho_\epsilon^\theta u^\theta_\epsilon )=0\\
  \partial_t u_\epsilon^\theta + u^\theta_\epsilon.\nabla_x u_\epsilon^\theta =E_\epsilon\\
  E_\epsilon=-\nabla_x V_\epsilon\\
  -\epsilon\Delta_x V_\epsilon = \int \int f_\epsilon^\theta dv d\mu -1\\
\end{array}
  \right.
\end{equation}
but the proofs are essentially the same and the study of system (\ref{toy}) has some interests of its own.

We consider global weak solutions to (\ref{toy}), in the sense of Arsenev \cite{Ar}.
We recall that the energy associated to (\ref{toy}) is the following non-increasing functional:
\begin{equation}
\mathcal{E}_\epsilon(t)  = \frac 1 2\int \int f_\epsilon^\theta \vert v\vert^2 dvdx d\mu + \frac 1 2\epsilon \int \vert \nabla_x V_\epsilon \vert^2 dx.
\end{equation}

We assume that there exists a constant $K>0$ independent of $\epsilon$, such as $\mathcal{E}_\epsilon(0) \leq K$. This implies that for any $\epsilon$ and $t>0$:
\begin{equation}
\mathcal{E}_\epsilon(t) \leq K.
\end{equation} 


Let $(\rho^\theta,u^\theta)$ be the local strong solution, to the system:
\begin{equation}
\left\{
    \begin{array}{ll}
  \partial_t \rho^\theta + \nabla_x.(\rho^\theta u^\theta)=0\\
  \partial_t u^\theta+ u^\theta.\nabla_x u^\theta =-\nabla_x V \\
 \int \rho^\theta d\mu =1.\\
\end{array}
  \right.
\end{equation}
with inital data $(\rho_0^\theta, u_0^\theta)$ (which we a priori have to take with analytic regularity).
The ``incompressibility'' constraint reads:
\begin{equation}
\nabla_x . \int \rho^\theta u^\theta d\mu =0.
\end{equation}

Following the approach of Brenier \cite{Br2} for the quasineutral limit with a single phase, we consider the relative entropy (built as a modulation of the energy $\mathcal{E}_\epsilon$):
\begin{equation}
\mathcal{H}_\epsilon(t)  = \frac 1 2 \int \int f_\epsilon^\theta \vert v-u^\theta(t,x)\vert^2 dvdx d\mu + \frac 1 2\epsilon \int \vert \nabla_x V_\epsilon - \nabla_x V\vert^2 dx.
\end{equation}

We assume that the system is well prepared in the sense that $\mathcal{H}_\epsilon(0)  \rightarrow 0$. The goal is to find some stability inequality in order to show that we also have $\mathcal{H}_\epsilon(t)  \rightarrow 0 $.

Our aim here, is to show why the method fails unless $u^\theta$ actually does not depend on $\theta$. This can be interpreted as the effect of the two-stream instabilities \cite{CGG}.

We have, since the energy is non-increasing:
\begin{equation}
\label{comp}
\begin{split}
\frac{d}{dt}\mathcal{H}_\epsilon(t)  \leq \int \int \partial_t f_\epsilon^\theta \left(\frac 1 2 \vert u^\theta\vert^2 -v .u^\theta\right)dvdx d\mu +  \int \int f_\epsilon^\theta \partial_t \left(\frac 1 2 \vert u^\theta \vert ^2 -v. u^\theta \right)dvdx d\mu \\
+ \frac 1 2\epsilon \int \partial_t \vert\nabla_x V\vert^2 dx - \epsilon \int \nabla_x V_\epsilon. \partial_t \nabla_x V dx
- \epsilon \int \partial_t \nabla_x V_\epsilon. \nabla_x V dx.
\end{split}
\end{equation}

We clearly have $ \epsilon \int \partial_t \vert\nabla_x V\vert^2 dx=\mathcal{O}(\epsilon)$. Moreover, we get, using the conservation of energy, 
\[
\epsilon \Big\vert \int \nabla_x V_\epsilon. \partial_t \nabla_x V dx \Big\vert \leq \sqrt{\epsilon} \Vert \sqrt \epsilon \nabla_x V_\epsilon\Vert_{L^2_x}  \Vert \partial_t \nabla_x V\Vert_{L^2_x},
\]
which is of order $\mathcal{O}(\sqrt \epsilon)$.

For the last term of (\ref{comp}), we compute:
\begin{equation}
\begin{split}
- \epsilon \int \partial_t \nabla_x V_\epsilon.\nabla_x V dx =&  \epsilon \int \partial_t \Delta_{x} V_\epsilon V dx \\
=& -\epsilon \int \partial_t \left(\int f_\epsilon^\theta dv d\mu\right) V dx \\
=& + \int \nabla_x.\left(\int f_\epsilon^\theta v dv d\mu\right) V dx \\
=&- \int \left(\int f_\epsilon^\theta v dv d\mu\right) \nabla_x V dx .
\end{split}
\end{equation}
In this computation we have used the local conservation of mass:
\[
\partial_t \int f_\epsilon^\theta dv + \nabla_x. \left( \int v f_\epsilon^\theta  dv\right) =0.
\]

In the other hand we can compute:
\begin{equation}
\begin{split}
&\int \int \partial_t f_\epsilon^\theta \left(\frac 1 2 \vert u^\theta \vert ^2 -v. u^\theta\right)dvdx d\mu +  \int \int f_\epsilon^\theta \partial_t \left(\frac 1 2 \vert u^\theta \vert ^2 -v. u^\theta\right)dvdx d\mu \\
=& \int \int f_\epsilon^\theta(u^\theta-v)  .(u^\theta-v).\nabla_x u_\theta dv dx d\mu +
\int\int f_\epsilon^\theta(u^\theta-v).(\partial_t u^\theta +u^\theta\nabla_x u^\theta) dvdx d\mu \\
 -& \int f_\epsilon^\theta E_\epsilon. u^\theta dv dx d\mu.
\end{split}
\end{equation}

All the trouble comes from this last term. When no assumption is made on $u^\theta$, it can be of order $\mathcal{O}(1/\sqrt{\epsilon})$. This wild term can be interpreted as the appearance of the two-stream instabilities.

Therefore we have to make an additional assumption in order to avoid this instability. This is done by assuming that $u^\theta$  initially does not depend on $\theta$ (which yields that $u_\theta$ does not depend on $\theta$ by uniqueness , in which case we can write:
\[
u_\theta=u 
\]
and consequently, we have
\begin{equation}
\begin{split}
- \int f_\epsilon^\theta E_\epsilon. u dv dx d\mu =& \int \left(\epsilon\Delta_x V_\epsilon -1\right)E_\epsilon. u dx .
\end{split}
\end{equation}

In addition, the incompressibility constraint becomes $\nabla_x.u=0$, and thus: 

$$\int E_\epsilon u dx
= \int V_\epsilon \nabla_x.u dx=0.$$

Furthermore, we have:

\begin{equation}
\begin{split}
\int \left(\int f_\epsilon^\theta dv d\mu\right) u. \nabla_x V dv =& \int u. \nabla_x V - \epsilon \int \Delta_x V_\epsilon u. \nabla_x V   \\
\end{split}
\end{equation}
The first term is equal to $0$ according to the incompressibility constraint, while the second is of order $\mathcal{O}({\sqrt{\epsilon}})$, by the energy inequality.

We finally get the stability inequality:
\begin{equation}
\begin{split}
\mathcal{H}_\epsilon(t) \leq \mathcal{H}_\epsilon(0) + R_\epsilon(t) + C\int_0^t \Vert \nabla_x u \Vert \mathcal{H}_\epsilon(s)ds\\
+ \int_0^t \int \int f_\epsilon^\theta (u-v) (\partial_t u + u.\nabla_x u +\nabla_x V) d\mu dv dx ds,
\end{split}
\end{equation}
with $R_\epsilon(t)\rightarrow 0$ as $\epsilon$ goes to $0$ and the last term is $0$ by definition of $u,V$.

As as result, by Gronwall's inequality, $\mathcal{H}_\epsilon(t) \rightarrow 0$, uniformly locally in time.

To conclude, we notice that $\rho^\theta_\epsilon:=\int f_\epsilon^\theta dv$ is uniformly bounded in $L^\infty_t(L^1_{\theta,x})$. By the energy inequality, we can easily show that this is also the case for $J_\epsilon^\theta:= \int f_\epsilon^\theta v dv$. Thus, up to a subsequence, there exist $\rho_\theta$ and $J^\theta$ such that $\rho_\epsilon^\theta$ weakly-* converges in the sense of measures to $\rho^\theta$ (resp. $J_\epsilon^\theta$ to $J^\theta$). Passing to the limit in the local conservation of charge, which reads:
\[
\partial_t \rho_\epsilon^\theta+ \nabla_x J_\epsilon^\theta=0,
\]
we obtain:
\[
\partial_t \rho^\theta+ \nabla_x J^\theta=0.
\]

The goal is now to prove that $J^\theta = \rho^\theta u$.

By a simple use of Cauchy-Schwarz inequality, we have:
\begin{equation}
\int \int \frac{\vert \rho_\epsilon^\theta u - J_\epsilon^\theta\vert}{\rho^\theta_\epsilon} dx d\mu \leq \int \int f_\epsilon^\theta\vert v- u\vert^2 dv dx d\mu.
\end{equation}

Using a classical convexity argument due to Brenier \cite{Br},  the functional $(\rho,J) \mapsto \int \frac{\vert \rho u-J\vert}{\rho} dxd\mu$ is lower semi-continuous with respect to the weak convergence of measures.  We finally obtain by passing to the limit that:
\[
J^\theta = \rho^\theta u.
\]

By uniqueness of the solution to the limit system, provided that  the whole sequence $\rho_{\epsilon,0}^\theta$ weakly converges to $\rho^\theta_0$, we obtain the convergences without having to extract subsequences.

Finally we have proved the result:

\begin{prop}
Let $(f_\epsilon^\theta,E_\epsilon)$ be a global weak solution to (\ref{toy}) such that the local conservation of charge and current are satisfied. Assume that for some smooth functions $(\rho_0^\theta, u_0)$ (we emphasize on the fact that $u_0$ does not depend on $\theta$, in order to avoid two-stream instabilities) satisfying:
\begin{equation}
\left\{
    \begin{array}{ll}
\int \rho_0^\theta d\mu=1, \\
\nabla_x. u_0=0,\\
\end{array}
  \right.
\end{equation}
we have:
\begin{equation}
 \frac 1 2 \int \int f_\epsilon( t=0)^\theta \vert v-u_0(x)\vert^2 dvdx d\mu + \frac 1 2\epsilon \int \vert \nabla_x V_\epsilon(t=0) - \nabla_x V\vert^2 dx \rightarrow 0
\end{equation}
and $\int f_\epsilon^\theta dv \rightharpoonup \rho_0^\theta$ in the weak sense of measures.
Then,
\begin{equation}
 \frac 1 2 \int \int f_\epsilon^\theta \vert v-u(t,x)\vert^2 dvdx d\mu + \frac 1 2\epsilon \int \vert \nabla_x V_\epsilon - \nabla_x V \vert^2 dx \rightarrow 0 ,
\end{equation}
where $(u,V)$ is the local strong solution to the incompressible Euler  system:
\begin{equation}
\left\{
    \begin{array}{ll}
  \partial_t u + u.\nabla_x u =-\nabla_x V \\
\nabla_x u=0.\\
\end{array}
  \right.
\end{equation}

Moreover, $\rho_\epsilon^\theta:=\int f_\epsilon^\theta dv$ converges in the weak sense of measures to $\rho^\theta$ solution to:
\begin{equation}
\partial_t \rho^\theta + u. \nabla_x \rho^\theta=0.
\end{equation}
and $J_\epsilon^\theta:= \int f_\epsilon^\theta v dv$ converges in the weak sense of measures to $\rho^\theta u$.
\end{prop}

\section{Conclusion}

In this work, we have provided a first analysis of the mathematical properties of the three-dimensional finite Larmor radius approximation (FLR), for electrons in a fixed background of ions. We have shown that the limit is unstable in the sense that we have to restrict to data with both particular profiles and analytic data. In particular we have pointed out that the analytic assumption is not only a mathematical technical assumption, but is necessary to have strong solutions. In addition the results are only local-in-time.

On the other hand, we proved in \cite{DHK1} that the FLR approximation for ions with massless electrons is by opposition very stable, in the sense that we can deal with initial data with no prescribed profile and weak (that is in a Lebesgue space) regularity.

This rigorously justifies why physicists rather consider the equations on ions rather than those on electrons, especially for numerical experiments (we refer for instance to Grandgirard et al. \cite{Gra}).

\section{Appendix : Formal derivation of the drift-fluid problem}

\label{formal}
\subsection*{Scaling of the Vlasov equation}

Let us recall that our purpose is to describe the behaviour of a gas of electrons in a neutralizing background of ions at thermodynamic equilibrium, submitted to a large magnetic field. For simplicity, we consider a magnetic field with a fixed direction $e_\parallel$ (also denoted by $e_z$) and a fixed large magnitude $\bar B$.


Because of the strong magnetic field, the dynamics of particles in the parallel direction $e_\parallel$  is completely different to their dynamics in the orthogonal plane. We therfore consider this time anisotropic characteristic spatial lengths:
$$\tilde{x}_\perp =\frac{x_\perp}{L_\perp}, \quad  \tilde{x}_\perp =\frac{x_\parallel}{L_\parallel,}$$

$$\tilde{t}= \frac t \tau \quad  \tilde{v}= \frac v {v_{th}},$$
$$f(t,x_\perp,x_\parallel,v)= \bar f \tilde{f}(\tilde{t},\tilde{x}_\perp,\tilde{x}_\parallel,\tilde{v}) \quad V(t,x_\perp,x_\parallel)= \bar V \tilde V(\tilde{t},\tilde{x}_\perp,\tilde{x}_\parallel) \quad E(t,x_\perp,x_\parallel)=\bar E \tilde E(\tilde t, \tilde{x}_\perp,\tilde{x}_\parallel). $$

The subscript $\perp$ stands for the orthogonal projection on the perpendicular plane (to the magnetic field), while the subscript $\parallel$ stands for the projection on the parallel direction.

This yields:

   \begin{equation}
\left\{
    \begin{array}{ll}
  \partial_{\tilde t} \tilde f_\epsilon +  \frac{v_{th}\tau}{L_\perp}\tilde v_\perp.\nabla_{\tilde x_\perp}\tilde f_\epsilon +  \frac{v_{th}\tau}{L_\parallel}\tilde v_\parallel.\nabla_{\tilde x_\parallel}\tilde f_\epsilon +\left(\frac{e\bar E \tau}{mv_{th}}\tilde E_\epsilon+ \frac{e \bar B}{m} \tau  \tilde v \wedge e_\parallel\right).\nabla_{\tilde v} \tilde f_\epsilon =0  \\
 \frac{\bar E }{\bar V} \tilde E_\epsilon = \left(-\frac{1}{L_\perp}\nabla_{\tilde x_\perp} \tilde V_\epsilon ,-\frac{1}{L_\parallel}\nabla_{\tilde x_\parallel} \tilde V_\epsilon\right)\\
-\frac{\epsilon_0\bar V}{L_\perp^2}\Delta_{\tilde x_\perp} \tilde V_\epsilon  -\frac{\epsilon_0\bar V}{L_\parallel^2}\Delta_{\tilde x_\parallel} \tilde V_\epsilon =e\bar f v_{th}^3\left(\int \tilde f_\epsilon d \tilde v - 1\right)\\
     \tilde f_{\epsilon,\vert \tilde t=0} =\tilde f_{0,\epsilon}, \quad \bar f L^3 v_{th}^3\int \tilde f_{0,\epsilon} d \tilde vd \tilde x =1.\\
    \end{array}
  \right.
\end{equation}

We set $\Omega=  \frac{e\bar B}{m}$ : this is the cyclotron frequency (also referred to as the gyrofrequency)
We also consider the so-called electron Larmor radius (or electron gyroradius) $r_L$ defined by:

\begin{equation}
r_L=\frac{ v_{th}}{\Omega}= \frac{m v_{th}}{e \bar B}\end{equation}
This quantity can be physically understood as the typical radius of the helix around axis $e_\parallel$ described by the particles, due to the intense magnetic field.

The Vlasov equation now reads:
$$ \partial_{\tilde t} \tilde f_\epsilon +  \frac{r_L}{L_\perp}\Omega \tau\tilde v_\perp.\nabla_{\tilde x_\perp}\tilde f_\epsilon +  \frac{r_L}{L_\parallel}\Omega \tau\tilde v_\parallel.\nabla_{\tilde x_\parallel}\tilde f_\epsilon +\left(\frac{\bar E }{\bar B v_{th}}\Omega \tau\tilde E_\epsilon+ \Omega \tau   \tilde v \wedge e_\parallel\right).\nabla_{\tilde v} \tilde f_\epsilon =0 . 
$$

The gyrokinetic ordering consists in:
$$\Omega \tau =\frac 1 \epsilon, \quad \frac{\bar E }{\bar B v_{th}}=\epsilon.$$

The spatial scaling we perform is the so-called finite Larmor radius scaling (see Fr\'enod and Sonnendrucker \cite{FS2} for a reference in the mathematical literature): basically the idea is to consider the typical perpendicular spatial length $L_\perp$ with the same order as the so-called electron Larmor radius.

On the contrary, the parallel observation length $L_\parallel$ is taken much larger:
\begin{equation}
\frac{r_L}{L_\perp} = 1 , \quad \frac{r_L}{L_\parallel} = \epsilon.  
\end{equation}
This is typically an anisotropic situation.

This particular scaling allows, at least in a formal sense, to observe more precise effects in the orthogonal plane than with the isotropic scaling (studied for instance in \cite{GSR1}):

$$
\frac{r_L}{L_\perp} = \epsilon , \quad \frac{r_L}{L_\parallel} = \epsilon.  $$

 In particular we wish to observe the so-called electric drift $E^\perp$ (also referred to as the $E \times B$ drift) whose effect is of great concern in tokamak physics (see \cite{DHK2} for instance).

The quasineutral ordering we adopt is the following:

\begin{equation}
\frac{\lambda_D} {L_\parallel} = \sqrt{\epsilon}.
\end{equation}

After straightforward calculations (we refer to \cite{FS2} for details), we get the following Vlasov-Poisson system in dimensionless form, for  $t\geq 0, x=(x_\perp,x_\parallel)  \in \mathbb{T}^2\times\mathbb{T}, v=(v_\perp,v_\parallel) \in\mathbb{R}^2\times \mathbb{R}$:

  \begin{equation}
\label{fix}
\left\{
    \begin{array}{ll}
    \partial_{t} f_\epsilon + \frac{v_\perp}{\epsilon}.\nabla_{x} f_\epsilon + v_\parallel.\nabla_{x} f_\epsilon + (E_\epsilon+ \frac{v\wedge e_z}{\epsilon}).\nabla_{v} f_\epsilon = 0 \\
  E_\epsilon= (-\nabla_{x_\perp} {V}_\epsilon, -\epsilon\nabla_{x_\parallel} {V}_\epsilon)  \\
  -\epsilon^2\Delta_{x_\parallel} {V}_\epsilon -\Delta_{x_\perp} {V}_\epsilon = \int f_\epsilon dv - \int f_\epsilon dvdx\\
     f_{\epsilon, t=0}=f_{\epsilon,0}.
\end{array}
  \right.
\end{equation}

\begin{rque}
It seems physically relevant to consider scalings such as:
\begin{equation}
\lambda_D / L_\parallel \sim \epsilon^\alpha,
\end{equation}
with $\alpha \geq 1$. However with such a scaling, the systems seem too degenerate with respect to $\epsilon$ and we have not been able to handle this situation. The scaling we study is nevertheless relevant for some extreme magnetic regimes in tokamaks .
\end{rque}

\subsection*{Hydrodynamic equations}

In order to isolate this quasineutral problem, thanks to the linearity of the Poisson equation, we separate the electric field in two parts:

  \begin{equation}
\label{model}
\left\{
    \begin{array}{ll}
E=  E^1_\epsilon +  E^2_\epsilon, \\
  E^1_\epsilon= (-\nabla_{x_\perp} {V}_\epsilon, -\epsilon\nabla_{x_\parallel} {V}_\epsilon),  \\
  -\epsilon^2\Delta_{x_\parallel} {V}^1_\epsilon -\Delta_{x_\perp} {V}^1_\epsilon = \int f_\epsilon dv - \int f_\epsilon dvdx_\perp,\\
E^2_\epsilon= - \partial_{x_\parallel} V^2_\epsilon, \\
  -\epsilon\Delta_{x_\parallel} {V}^2_\epsilon= \int f_\epsilon dvdx_\perp - \int f_\epsilon dv dx. \\
\end{array}
  \right.
\end{equation}

In order to make the fast oscillations in time due to the term $\frac{v_\perp}{\epsilon}.\nabla_{x}$ disappear, we perform the same change of variables as in \cite{GHN}, to get the so-called gyro-coordinates:
\begin{equation}
x_g= x_\perp + v^\perp, v_g=v_\perp.
\end{equation}
We easily compute the equation satisfied by the new distribution function $g_\epsilon(t,x_g,v_g,v_\parallel)=f_\epsilon(t,x,v)$. 
\begin{eqnarray*}
 \partial_t g_\epsilon + v_\parallel \partial_{x_\parallel} g_\epsilon + E^1_{\epsilon, \parallel}(t,x_g-v_g^\perp)\partial_{v_\parallel} g_\epsilon +  E^2_{\epsilon}(t,x_{g,\parallel})\partial_{v_\parallel} g_\epsilon \\
+E^1_{\epsilon, \perp}(t,x_g- v_g^\perp).(\nabla_{v_g} g_\epsilon - \nabla^\perp_{x_g} g_\epsilon) + \frac{1}{\epsilon} v_g^\perp. \nabla_{v_g} g_\epsilon=0.
\end{eqnarray*}
Notice here that in the process, the so-called electric drift $E^\perp$ appears since:
$$
-E^1_{\epsilon, \perp}(t,x_g- v_g^\perp). \nabla^\perp_{x_g} g_\epsilon = E^{1,\perp}_{\epsilon}(t,x_g- v_g^\perp).  \nabla_{x_g} g_\epsilon.
$$

The equation satisfied by the charge density $\rho_\epsilon=\int g_\epsilon dv$ states:
\begin{equation}
 \partial_t \rho_\epsilon + \partial_{x_\parallel} \int v_\parallel g_\epsilon dv + \nabla^\perp_{x_g}\int E^1_{\epsilon, \perp}(t,x_g- v_g^\perp) g_\epsilon dv = 0,
\end{equation}
and the one satisfied by the current density $J_\epsilon=\int g_\epsilon v dv\left(= \begin{pmatrix} \int g_\epsilon v_\perp dv \\ \int g_\epsilon v_\parallel dv \end{pmatrix}\right)$ is the following:
\begin{eqnarray}
 \partial_t J_\epsilon + \partial_{x_\parallel} \int v_\parallel \begin{pmatrix} v_g \\ v_\parallel \end{pmatrix} g_\epsilon dv + \nabla^\perp_{x_g}\int E^1_{\epsilon, \perp}(t,x_g- v_g^\perp) \begin{pmatrix} v_g \\ v_\parallel \end{pmatrix} g_\epsilon dv  \nonumber \\ 
=\int \begin{pmatrix} E^1_{\epsilon, \perp}(t,x_g- v_g^\perp)  \\ 0 \end{pmatrix} g_\epsilon dv + \int    \begin{pmatrix} 0 \\ E^1_{\epsilon, \parallel}(t,x_g-v_g^\perp) \end{pmatrix}g_\epsilon dv \nonumber \\
 + \begin{pmatrix} 0 \\ E^2_{\epsilon}(t,x_{g,\parallel})\rho_\epsilon \end{pmatrix} + \frac{J_\epsilon^\perp}{\epsilon}.
\end{eqnarray}

We now assume that we deal with special monokinetic data of the form:
\begin{equation}
 g_\epsilon(t,x,v) = \rho_\epsilon(t,x)\mathbbm{1}_{v_\parallel=v_{\parallel,\epsilon}(t,x)}\mathbbm{1}_{v_g=0}.
\end{equation}

This assumption  is nothing but the classical ``cold plasma'' approximation together with the assumption that the transverse particle velocities are isotropically distributed (which is physically relevant, see \cite{Sul}) : in other words, the average motion of particles in the perpendicular plane is only due to the advection by the electric drift $E^\perp$. 


For the sake of readability, we denote by now  $\nabla_{x_g}= \nabla_\perp$ and $\nabla_{x_\parallel}= \nabla_\parallel$.
Then we get formally the hydrodynamic model:

\begin{equation}
 \left\{
    \begin{array}{ll}
  \partial_t \rho_\epsilon + \nabla_\perp (E^\perp_\epsilon \rho_\epsilon) + \partial_\parallel(v_{\parallel,\epsilon} \rho_\epsilon)= 0 \\
  \partial_t (\rho_\epsilon v_{\parallel,\epsilon}) + \nabla_\perp (E^\perp_\epsilon \rho_\epsilon v_{\parallel,\epsilon}) + \partial_\parallel(\rho_\epsilon v_{\parallel,\epsilon}^2) = -\epsilon\partial_\parallel \phi_\epsilon(t,x)\rho_\epsilon -\partial_\parallel V_\epsilon(t,x_\parallel)\rho_\epsilon \\
  E^\perp_\epsilon= -\nabla^\perp \phi_\epsilon \\
-\epsilon^2 \partial^2_{\parallel} \phi_\epsilon - \Delta_{\perp} \phi_\epsilon = \rho_\epsilon - \int \rho_\epsilon dx_\perp\\
  -\epsilon \partial_{\parallel}^2 V_\epsilon =\int \rho_\epsilon dx_\perp -1\\
\end{array}
  \right.
\end{equation}

One can use the first equation to simplify the second one (the systems are equivalent provided that we work with regular solutions and that $\rho_\epsilon>0$):

\begin{equation}
 \left\{
    \begin{array}{ll}
  \partial_t \rho_\epsilon + \nabla_\perp (E^\perp_\epsilon \rho_\epsilon) + \partial_\parallel(v_{\parallel,\epsilon} \rho_\epsilon)= 0 \\
  \partial_t v_{\parallel,\epsilon} + \nabla_\perp (E^\perp_\epsilon v_{\parallel,\epsilon}) + v_{\parallel,\epsilon} \partial_\parallel(v_{\parallel,\epsilon}) = -\epsilon\partial_\parallel \phi_\epsilon(t,x) -\partial_\parallel V_\epsilon(t,x_\parallel) \\
  E^\perp_\epsilon= -\nabla^\perp \phi_\epsilon \\
-\epsilon^2 \partial^2_{\parallel} \phi_\epsilon - \Delta_{\perp} \phi_\epsilon = \rho_\epsilon - \int \rho_\epsilon dx_\perp\\
  -\epsilon \partial_{\parallel}^2 V_\epsilon =\int \rho_\epsilon dx_\perp -1.\\
\end{array}
  \right.
\end{equation}

\begin{rques}\begin{enumerate}
\item  Notice here that we do not deal with the usual charge density and current density, since these ones are taken within the gyro-coordinates.

 \item We mention that we could have considered the more general case:
\begin{equation}
\label{gene}
 g_\epsilon(t,x,v) = \int_ M \rho_\epsilon^\Theta (t,x)\mathbbm{1}_{v_\parallel=v^\Theta_{\parallel,\epsilon}(t,x)} \nu(d\Theta) \mathbbm{1}_{v_g=0}
\end{equation}
where $(M,\Theta,\nu)$ is a probability space which allows to model more realistic plasmas than ``cold plasmas'' and covers many interesting physical data, like multi-sheet electrons or water-bags data (we refer for instance to \cite{BBBB} and references therein). We will not do so for the sake of readability but we could deal with it with exactly the same analytic framework: the analogues of Theorems \ref{exi} and \ref{con} identically hold. We get in the end the system:

\begin{equation}
\label{sysfinal}
\left\{
    \begin{array}{ll}
  \partial_t \rho^\Theta + \nabla_\perp (E^\perp \rho^\Theta) + \partial_\parallel(v^\Theta_\parallel \rho^\Theta)= 0 \\
  \partial_t v^\Theta_\parallel + \nabla_\perp (E^\perp v^\Theta_\parallel) + v^\Theta_\parallel \partial_\parallel(v^\Theta_\parallel) = -\partial_\parallel p(t,x_\parallel) \\
  E^\perp= \nabla^\perp \Delta_\perp^{-1} \left(\int \rho^\Theta d\nu - \int \rho^\Theta dx_\perp d\nu\right)\\
  \int \rho^\Theta(t,x) dx_\perp d\nu = 1.\\
\end{array}
  \right.
\end{equation}
As before, the equations are coupled through $x_\perp$ and here also through the new parameter $\Theta$.
 
 \item Actually, the choice:
\begin{equation}
 g_\epsilon(t,x,v) = \rho_\epsilon(t,x)\mathbbm{1}_{v=v_{\epsilon}(t,x)}
 \end{equation}
 leads to an ill-posed system. Indeed, we have to solve in this case equations of the form $v^\perp_\epsilon=v_{\epsilon,\perp}(t,x-v^\perp_\epsilon)$ where $v_{\epsilon,\perp}$ is the unknown. We can not say if this relation is invertible, even locally.

\end{enumerate}
\end{rques}

\begin{ack}
This work originated from discussions with Maxime Hauray; I would like to thank him for that.
\end{ack}

\bibliographystyle{plain}
\bibliography{Larmoranalytic}

\end{document}